\documentclass[11pt]{amsart}
\usepackage{amssymb}

\makeatletter
\newcommand{\sumprime}{\if@display\sideset{}{'}\sum%
            \else\sum'\fi}
\makeatother

\begin{document}

\numberwithin{equation}{section}

\newtheorem{theorem}{Theorem}[section]
\newtheorem{proposition}[theorem]{Proposition}
\newtheorem{conjecture}[theorem]{Conjecture}
\def\theconjecture{\unskip}
\newtheorem{corollary}[theorem]{Corollary}
\newtheorem{lemma}[theorem]{Lemma}
\newtheorem{observation}[theorem]{Observation}
\newtheorem{definition}{Definition}
\numberwithin{definition}{section} 
\newtheorem{remark}{Remark}
\def\theremark{\unskip}
\newtheorem{question}{Question}
\def\thequestion{\unskip}
\newtheorem{example}{Example}
\def\theexample{\unskip}
\newtheorem{problem}{Problem}

\def\vvv{\ensuremath{\mid\!\mid\!\mid}}
\def\intprod{\mathbin{\lr54}}
\def\reals{{\mathbb R}}
\def\integers{{\mathbb Z}}
\def\N{{\mathbb N}}
\def\complex{{\mathbb C}\/}
\def\dist{\operatorname{dist}\,}
\def\spec{\operatorname{spec}\,}
\def\interior{\operatorname{int}\,}
\def\trace{\operatorname{tr}\,}
\def\cl{\operatorname{cl}\,}
\def\essspec{\operatorname{esspec}\,}
\def\range{\operatorname{\mathcal R}\,}
\def\kernel{\operatorname{\mathcal N}\,}
\def\dom{\operatorname{Dom}\,}
\def\linearspan{\operatorname{span}\,}
\def\lip{\operatorname{Lip}\,}
\def\sgn{\operatorname{sgn}\,}
\def\Z{ {\mathbb Z} }
\def\e{\varepsilon}
\def\p{\partial}
\def\rp{{ ^{-1} }}
\def\Re{\operatorname{Re\,} }
\def\Im{\operatorname{Im\,} }
\def\dbarb{\bar\partial_b}
\def\eps{\varepsilon}
\def\O{\Omega}
\def\Lip{\operatorname{Lip\,}}

\def\Hs{{\mathcal H}}
\def\E{{\mathcal E}}
\def\scriptu{{\mathcal U}}
\def\scriptr{{\mathcal R}}
\def\scripta{{\mathcal A}}
\def\scriptc{{\mathcal C}}
\def\scriptd{{\mathcal D}}
\def\scripti{{\mathcal I}}
\def\scriptk{{\mathcal K}}
\def\scripth{{\mathcal H}}
\def\scriptm{{\mathcal M}}
\def\scriptn{{\mathcal N}}
\def\scripte{{\mathcal E}}
\def\scriptt{{\mathcal T}}
\def\scriptr{{\mathcal R}}
\def\scripts{{\mathcal S}}
\def\scriptb{{\mathcal B}}
\def\scriptf{{\mathcal F}}
\def\scriptg{{\mathcal G}}
\def\scriptl{{\mathcal L}}
\def\scripto{{\mathfrak o}}
\def\scriptv{{\mathcal V}}
\def\frakg{{\mathfrak g}}
\def\frakG{{\mathfrak G}}

\def\ov{\overline}

\thanks{Supported by Grant IDH1411001 from Fudan University}

\address{School of Mathematical Sciences, Fudan University, Shanghai 200433, China}
 \email{boychen@fudan.edu.cn}

\title{ Parameter dependence of the Bergman kernels}
\author{Bo-Yong Chen}
\date{}
\maketitle

\bigskip

\begin{abstract}
Let $\{\Omega_t:-1<t<1\}$  be a family of bounded pseudoconvex domains and $\varphi_t\in PSH(\Omega_t)$. Let $K_t(z,w)$ denote the Bergman kernel with weight $\varphi_t$ on $\Omega_t$. We study the continuity and H\"older continuity of $K_t(z,w)$ in $t$. Several applications to singularity theory of psh functions are given, including a new proof of the openness theorem.
\end{abstract}

\section{Introduction}

Let $\{\Omega_t:|t|<1\}$ ($t\in {\mathbb R}$ or $t\in {\mathbb C}$) be a family of bounded domains in ${\mathbb C}^n$ and $\varphi_t\in PSH(\Omega_t):$ the set of plurisubharmonic (psh) functions on $\Omega_t$. Let $K_t(z,w)$ denote the Bergman kernel corresponding to the Hilbert space
$$
A^2(\Omega_t,{\varphi_t}) :=\left\{f\in {\mathcal O}(\Omega_t): \int_{\Omega_t} |f|^2 e^{-\varphi_t}<\infty\right\}.
$$
There are two general  approaches to study the parameter dependence of $K_t$: (1) regularity of $K_t$ in $t$; (2) convexity or (pluri)subharmonicity of $K_t$ in $t$. It is known from the works of Hamilton \cite{Hamilton79} and Greene-Krantz \cite{GreeneKrantz82} that $K_t$ is $C^\infty$ in $t$ when $\{\Omega_t\}$ is a family of strongly pseudoconvex domains such that $\{\partial\Omega_t\}$ forms a differentiable family of compact manifolds, and $\varphi_t=0$ for all $t$. Little is known about the case of  weakly pseudoconvex  domains or when $\varphi_t$ has\/ {\it singularities}. On the other side, the second approach is by now well-developed through a series of papers due to Berndtsson after the seminal work of Maitani-Yamaguchi \cite{MaitaniYamaguchi}, which turns out to be very useful in complex analysis and complex geometry (see e.g., \cite{BerndtssonSubharmonic}, \cite{BerndtssonCurvature}, \cite{BerndtssonOpenness}).

This paper is closer to the first approach. We consider the following two special cases:
\begin{enumerate}
 \item $\{\varphi_t:-1<t<1\}$ is a family of negative psh functions on a fixed domain $\Omega$.
 \item $\{\Omega_t:-1<t<1\}$ is a family of bounded domains and $\varphi_t=0$ for all $t$.
\end{enumerate}

Let $PSH^-(\Omega)$ denote the set of negative psh functions on $\Omega$.

\begin{definition}
We say that a sequence $\{\varphi_j\}\subset PSH^-(\Omega)$ satisfies condition $(\ast)$ if there exists a closed complete pluripolar set $E\subset \Omega$ such that for every compact set $S\subset \Omega\backslash E$ there is a positive function $\phi_S\in L^1(S)$ satisfying $e^{-\varphi_j}\le \phi_S$ on $S$ for sufficiently large $j$.
\end{definition}

Here a complete pluripolar set $E$ means that for every $a\in E$ there exist a neighborhood $U$ of $0$ and a nonconstant function $\psi\in PSH(U)$ such that $E\cap U=\psi^{-1}(-\infty)$.

\begin{example}[1]
Consider a family $\{\psi_t:-1<t<1\}\subset PSH^-(\Omega)$ such that $e^{\psi_t(z)}$ is continuous in $(z,t)\in {\Omega} \times (-1,1)$. Set $E:=\psi_0^{-1}(-\infty)$ and $\varphi_j=\psi_{1/j}$. Clearly, for every compact set $S\subset \Omega\backslash E$, $e^{-\varphi_j}$ is bounded by a positive constant on $S$ for all sufficiently large $j$, so that $\{\varphi_j\}$ satisfies condition $(\ast)$. We may choose for instance $\psi_t(z)=\alpha(t)\log \sum_j |f_j(z,t)|^2$ where $f_j(z,t)\in C(\Omega\times (-1,1))$, $1\le j\le m$, $f_j(\cdot,t)\in {\mathcal O}(\Omega)$ with $|f_j|\ll 1$, and $\alpha\in C((-1,1))$ with $\alpha(t)\ge c>0$ for all $t$.
\end{example}

\begin{example}[2]
Suppose that $\psi\in PSH^-(\Omega)$. Set $\varphi_t=t\psi$, $t>0$. Fix $c>0$. Set
$$
E:=\left\{z\in \Omega:e^{-c\psi}\ { is\ not\ }L^1\ { in\ any\ neighborhood\ of\ }z\right\}.
$$
By virtue of Bombieri's theorem (cf. \cite{HormanderBook}, Corollary 4.4.6),  $E$ is an analytic subset in $\Omega$, hence is a closed complete pluripolar set.
On the other hand, $e^{-c\psi}\in L^1(\Omega\backslash E,{\rm loc})$. If we set $\phi_S=e^{-c\psi}$ for every compact set $S\subset \Omega\backslash E$, then for every sequence $t_j\rightarrow t_0< c$, $\{\varphi_{t_j}\}$ satisfies condition $(\ast)$.
\end{example}

A domain $\Omega\subset {\mathbb C}^n$ is called hyperconvex if there exists a continuous function $\rho\in PSH^-(\Omega)$ such that $\{\rho<c\}\subset\subset \Omega$ for every $c<0$.

 \begin{theorem}\label{th:weighted_Continue}
        Let $\Omega\subset {\mathbb C}^n$ be a bounded hyperconvex domain. Suppose that $\{\varphi_j\}\subset PSH^-(\Omega)$ satisfies condition $(\ast)$ and $\varphi_j$ converges almost everywhere on $\Omega$ to a function $\varphi\in PSH^-(\Omega)$. Let $K_j$ and $K$ denote the Bergman kernel with weight $\varphi_j$ and $\varphi$ on $\Omega$. Then $K_j(z,w)$ converges locally uniformly to $K(z,w)$ on $\Omega\times \Omega$.
       \end{theorem}

        \begin{corollary}\label{coro:weighted_Continue}
        Let $\Omega\subset {\mathbb C}^n$ be a bounded hyperconvex domain and $\varphi_t\in PSH^-(\Omega)$, $-1<t<1$. Let $K_t$ denote the Bergman kernel with weight $\varphi_t$ on $\Omega$. Suppose  $e^{\varphi_t(z)}$ is continuous in $(z,t)\in {\Omega} \times (-1,1)$. Then $K_t(z,w)$ is continuous in $t$.
       \end{corollary}

       The proof of Theorem \ref{th:weighted_Continue} relies heavily on the $L^2-$estimates of Donnelly-Fefferman (cf. \cite{DonnellyFefferman}, see also \cite{Berndtsson01}). The key observation is an approximation result for holomorphic functions (see Lemma \ref{lm:keyLemma}), which also has applications in singularity theory of psh functions, including a new proof of Berndtsson's openness theorem (cf. \cite{BerndtssonOpenness}).

    In order to study the H\"older continuous parameter dependence of the weighted Bergman kernels, we need two fundamental concepts from singularity theory of psh functions.

           \begin{definition}[see e.g., \cite{DemaillyKollar}]
   Let $\varphi$ be a psh function in a neighborhood of\/ $0$. The log canonical threshold\/ $($or complex singularity exponent\,$)$ $c_0(\varphi)$ of $\varphi$ at\/ $0$ is defined as
   $$
   c_0(\varphi):=\sup\{c\ge 0: e^{-c\varphi}\ { is\ } L^1\ { in \ a\ neighborhood\ of\ } 0\}.
   $$
  \end{definition}

   \begin{definition}
 The {\L}ojasiewicz exponent of a psh function $\varphi$ with an isolated singularity at\/ $0$ is defined as
$$
\mathcal L_0(\varphi)  =  \inf \left\{c\ge 0: e^{\varphi(z)}\ge {\rm const}_c\, |z|^c { \ in\ a \ neighborhood\ of\ } 0\right\}.
$$
\end{definition}

By convention, we set $\mathcal L_0(\varphi) =\infty$ if the previous set is empty.

  \begin{theorem}\label{th:Weighted_Holder}
 Let $\Omega$ be a bounded pseudoconvex domain with\/ $0\in \Omega$. Let $\{\varphi_t:-1<t<1\}$ be a family of negative psh functions on ${\Omega}$ such that
  $e^{\varphi_0}$ is a continuous function with an isolated zero at\/ $0$, $c_0(\varphi_0)>1$,  and
 $$
 |e^{\varphi_t(z)}-e^{\varphi_0(z)}|\le C |t|^\alpha,\ \ \ z\in \Omega,
 $$
 where $C>0$ and $0<\alpha\le 1$. Let $K_t$ denote the Bergman kernel with weight $\varphi_t$ on $\Omega$. Then
  \begin{enumerate}
  \item $K_t(w)$ is H\"older continuous of order $\beta$ at $t=0$ for every $\beta<\frac{c_0(\varphi_0)-1}{c_0(\varphi_0)+1}\alpha$, and every  $w\in \Omega\backslash \{0\}$.
   \item $K_t(0)$ is H\"older continuous of order $\beta$ at $t=0$ for every
   $
   \beta<\frac{\eta_0}{1+\eta_0\tau_0}\alpha,
   $
   where
   $$
    \eta_0=\min\left\{\frac{1}{\mathcal L_0(\varphi)},\frac{c_0(\varphi_0)-1}{2n}\right\},\ \ \ \tau_0=\min\left\{\frac{c_0(\varphi_0)-1}{2\eta_0}-n,1\right\}.
   $$
  \end{enumerate}
\end{theorem}

\begin{remark}
Notice that one can choose $\beta$ arbitrarily close to $\alpha$ in case\/ $(1)$, provided $c_0(\varphi_0)$ sufficiently large.
\end{remark}

\begin{definition}\label{def:Holder}
 Let $\left\{\Omega_t:-1<t<1\right\}$ be a family of domains in ${\mathbb C}^n$. Let $\rho$ be a negative continuous function on the total set
 $$
 \Omega=\{(z,t):z\in \Omega_t,t\in (-1,1)\}
 $$
 which satisfies
 $
 \{-\rho_t> \varepsilon\}\subset\subset \Omega_t
 $
  where $\rho_t=\rho(\cdot,t)$, for $\varepsilon>0$ and $t\in (-1,1)$. We say that $\Omega_t$ is $\rho_t-$H\"older continuous of order $\alpha$ over $(-1,1)$ if for each $\gamma>0$ there exist positive numbers $b_\gamma\gg 1$ and $c_\gamma\ll1$ such that
   $$
   \{-\rho_t> b_\gamma\,|t-s|^\alpha\}\subset  \{-\rho_s> \gamma\, |t-s|^\alpha\}
   $$
    for all $t,s\in (-1,1)$ with $|t-s|\le c_\gamma$.
\end{definition}

 Our main result is the following

  \begin{theorem}\label{th:principle}
        Let $\left\{\Omega_t:-1<t<1\right\}$ be a family of bounded domains in ${\mathbb C}^n$. Suppose there exists for every $t\in (-1,1)$ a negative continuous psh exhaustion function $\rho_t$ on $\Omega_t$ such that $\Omega_t$ is $\rho_t-$H\"older continuous of order $\alpha$ over $(-1,1)$.
         Then the Bergman kernel $K_t(z,w)$ of\/ $\Omega_t$ is H\"older continuous of order $\beta$ in $t$ for every $\beta<\alpha$.
                \end{theorem}

  As a direct consequence, we obtain

  \begin{corollary}\label{coro:exhaustion}
   Let $\Omega\subset {\mathbb C}^n$ be a pseudoconvex domain and $\rho$ a continuous psh exhaustion function on $\Omega$. Let $\Omega_t:=\{z\in \Omega:\rho(z)<t\}$, $t\in {\mathbb R}$. Then  $K_t(z,w)$ is H\"older continuous of order $\alpha$ in $t$ for every $\alpha<1$.
  \end{corollary}

  We also study in \S\,7 the (optimal) H\"older continuity of $K_t$ in $t$ for a $(-\delta_t)-$H\"older continuous family $\{\Omega_t\}$ of bounded simply-connected planar domains, where $\delta_t$ denotes the boundary distance of $\Omega_t$.

        Diederich-Ohsawa \cite{DidOhsParameter} studied the continuous parameter dependence of the $L^2-$minimal solutions of the $\bar{\partial}-$equations with respect to certain psh weight functions. It would be interesting to know whether similar H\"older continuity holds for the (unweighted) $L^2-$minimal solutions of the $\bar{\partial}-$equations under situations considered here.

         For the proof of Theorem \ref{th:principle}, we use  a nice weighted estimate of the $L^2-$minimal solution of the $\bar{\partial}-$equation due to Berndtsson, together with certain iteration procedure.

\section{Weighted estimates for the $L^2-$minimal solution of the $\bar{\partial}-$equation}

Let $\Omega\subset {\mathbb C}^n$ be a bounded pseudoconvex domain and let $\varphi\in PSH(\Omega)$. By H\"ormander's $L^2-$existence theorem for
the $\bar{\partial}-$equation (cf. \cite{HormanderBook}), we know that for every $\bar{\partial}-$closed $(0,1)-$form $v$ on $\Omega$ with
 $
\int_\Omega |v|^2 e^{-\varphi}<\infty,
$
there exists a solution $u$ to $\bar{\partial}u=v$ such that
$$
\int_\Omega |u|^2 e^{-\varphi} \le C \int_\Omega |v|^2 e^{-\varphi}
$$
where $C>0$ is a constant depending only on $n$ and ${\rm diam}(\Omega)$.
                Let $L^2(\Omega,\varphi)$ denote the Hilbert space of measurable functions $f$ satisfying
                 $$
                 \|f\|^2:= \int_\Omega |f|^2 e^{-\varphi}<\infty.
                 $$
                 We say that $u$ is the (unique) $L^2(\Omega,\varphi)-$minimal solution of the $\bar{\partial}-$equation if $u\bot A^2(\Omega,\varphi)$ in $L^2(\Omega,\varphi)$, i.e., $u$ has minimal norm $\|\cdot\|$ among all solutions.

                 Berndtsson proved that the $L^2(\Omega,\varphi)-$minimal solution satisfies the following estimate which goes back to Donnelly-Fefferman \cite{DonnellyFefferman}.

                 \begin{theorem}[cf. \cite{Berndtsson01}]\label{th:Berndtsson-Donnelly-Fefferman}
Let $\Omega\subset {\mathbb C}^n$ be a bounded pseudoconvex domain and $\varphi$ a $C^2$ psh function on $\Omega$. Suppose $\psi$ is a $C^2$  real function satisfying
\begin{equation}\label{eq:DFCondition}
ri\partial\bar{\partial}(\varphi+\psi)\ge i\partial\psi\wedge \bar{\partial}\psi
\end{equation}
for some $0<r<1$. Then the $L^2(\Omega,\varphi)-$minimal solution of $\bar{\partial}u=v$ satisfies
\begin{equation}\label{eq:Berndtsson-DF}
\int_\Omega |u|^2 e^{\psi-\varphi} \le \frac6{(1-{r})^2} \int_\Omega  |v|_{i\partial\bar{\partial}(\varphi+\psi)}^2 e^{\psi-\varphi}.
\end{equation}
\end{theorem}

He also proved the following

\begin{theorem}[cf. \cite{Berndtsson97}]\label{th:Berndtsson97}
 Let $\Omega$ be a bounded pseudoconvex domain and $\varphi\in PSH(\Omega)$.  Let $u$ be the $L^2(\Omega,\varphi)-$minimal solution of $\bar{\partial}u=v$. Let $\omega$ be a positive continuous $(1,1)-$form on $\Omega$. Then
 $$
 \int_\Omega |u|^2 e^{-\varphi}\Psi\le \int_\Omega |v|^2_\omega e^{-\varphi}\Psi
 $$
 for all $C^2$ positive functions $\Psi$ on $\Omega$ such that
 $$
 i\partial\bar{\partial}\Psi\le \Psi (i\partial\bar{\partial}\varphi-\omega).
 $$
\end{theorem}

As a direct consequence, we obtain
\begin{corollary}
Let $\Omega$ be a bounded pseudoconvex domain and $\varphi\in PSH(\Omega)$. Let $\psi$ be a $C^2$ psh function on $\Omega$
 which satisfies $ri\partial\bar{\partial}\psi\ge i\partial\psi\wedge \bar{\partial}\psi$ for some $0<r<1$. Then  the $L^2(\Omega,\varphi)-$minimal solution satisfies
 \begin{equation}\label{eq:L2Minimal}
 \int_\Omega |u|^2 e^{-\psi-\varphi}\le \frac{1}{1-r} \int_\Omega |v|^2_{i\partial\bar{\partial}\psi} e^{-\psi-\varphi}.
 \end{equation}
\end{corollary}

\begin{proof}
Set
$
\Psi=e^{-\psi}
$
and
$$
\omega=(1-r)i\partial\bar{\partial}\psi.
$$
We then have
$$
i\partial\bar{\partial}\Psi=\Psi(i\partial\psi\wedge \bar{\partial}\psi-i\partial\bar{\partial}\psi)\le -\Psi \omega,
$$
so that Theorem \ref{th:Berndtsson97} applies.
\end{proof}

\begin{remark}
 Following a suggestion of Blocki \cite{BlockCRvsMA}, we may deal with the case when $\psi$ is not $C^2$:  $|v|^2_{i\partial\bar{\partial}\psi}$ should be replaced by\/ {\rm any\/} non-negative locally bounded function $H$ such that
 $$
 i \bar{v}\wedge v\le H i\partial\bar{\partial}\psi
 $$
 holds in the sense of distributions.
 This is very convenient for various applications.
\end{remark}

\section{Proof of Theorem \ref{th:weighted_Continue}}

\begin{proposition}\label{prop:Upper_Semicontinuity}
 Let $U$ be a domain in ${\mathbb C}^n$ and $\{\varphi_j\}$ a sequence of non-positive psh functions on $U$ such that $\varphi_j\rightarrow \varphi$ a.e. on $U$. Let $K_j$ and $K$ denote the Bergman kernel with weight $\varphi_j$ and $\varphi$ respectively. Then
 $$
{\lim\sup}_{j\rightarrow \infty} K_j(z)\le  K(z),\ \ \ z\in U.
$$
\end{proposition}

\begin{proof}
Fix a compact set $S\subset U$ and a point $w\in S$ for a moment. Suppose
$$
K_{j_k}(w)\rightarrow {\lim\sup}_{j\rightarrow \infty} K_j(w)
$$
 as $k\rightarrow \infty$. Set $f_k(z)=K_{j_k}(z,w)$. For every $k$ and $h\in {\mathcal O}(U)$ with
$$
\int_\Omega |h|^2 e^{-\varphi_{j_k}}= 1,
$$
we have
$
\int_\Omega |h|^2 \le 1,
$
so that
$
|h(w)|^2\le {\rm const}_S
$
in view of the mean value inequality.  It follows immediately that
$
K_{j_k}(w)\le {\rm const}_S.
$
Since
$$
\int_U |f_{k}|^2 \le \int_U |f_{k}|^2 e^{-\varphi_{j_k}}=K_{j_k}(w)\le {\rm const}_S,
$$
so there exists a subsequence which is still denoted by $\{f_{k}\}$, such that $f_{k}\rightarrow f\in {\mathcal O}(U)$ locally uniformly. Fatou's lemma yields
\begin{eqnarray*}
\int_U |f|^2 e^{-\varphi} & \le & {\lim\inf}_{k\rightarrow \infty} \int_U |f_k|^2 e^{-\varphi_{j_k}}\\
& = & {\lim}_{k\rightarrow \infty} K_{j_k}(w)\\
& = & {\lim\sup}_{j\rightarrow \infty} K_j(w).
\end{eqnarray*}
Since $f(w)=\lim_{k\rightarrow \infty} f_k(w)={\lim\sup}_{j\rightarrow \infty} K_j(w)$, so we have
$$
K(w)\ge \frac{|f(w)|^2}{\|f\|^2_{L^2(U,\varphi)}}\ge {\lim\sup}_{j\rightarrow \infty} K_j(w).
$$
\end{proof}

\begin{lemma}\label{lm:simple}
 Let $U$ be a bounded hyperconvex domain and $\rho$ a negative continuous psh exhaustion function on $U$. Set $U_\varepsilon=\{\rho<-\varepsilon\}$ for $\varepsilon>0$. Let $\varphi\in PSH^-(U)$. Let $S$ be a compact set in $U$. For every $f\in A^2(U_\varepsilon,\varphi)$ and $w\in S$, there exists $g\in A^2(U,\varphi)$ satisfying $g(w)=f(w)$ and
 $$
\|g\|_{L^2(U,\varphi)}\le (1+{\rm const}_S/|\log \varepsilon|)\|f\|_{L^2(U_\varepsilon,\varphi)}
$$
provided $\varepsilon\le\varepsilon_S\ll 1$.
\end{lemma}

\begin{proof}
 Without loss of generality, we assume $-\rho<1$. Let $\chi:{\mathbb R}\rightarrow [0,1]$ be a smooth function satisfying $\chi|_{(0,\infty)}=0$ and $\chi|_{(-\infty,-\log 2)}=1$.
 Set
$$
\lambda_\varepsilon=\chi(\log(-\log(-\rho))-\log(-\log\varepsilon)).
$$
Applying Theorem \ref{th:Berndtsson-Donnelly-Fefferman} with $\psi$ and $\varphi$ replaced by $-\frac12 \log(-\rho)$ and $\varphi+2n\log|z-w|-\frac12 \log(-\rho)$ respectively, we then obtain a solution $u_{\varepsilon}$ of  $\bar{\partial} u= f \bar{\partial}\lambda_\varepsilon$ on $U$ satisfying
\begin{eqnarray*}
\int_U |u_{\varepsilon}|^2 e^{-\varphi-2n\log |z-w|}
 & \le & 24 \int_U |f|^2 |\bar{\partial}\lambda_{\varepsilon}|^2_{ -\frac{i}2\partial\bar{\partial}\log(-\rho)}e^{-\varphi-2n\log |z-w|}\\
& \le & {\rm const}_S\,  |\log \varepsilon|^{-2}\int_{U_\varepsilon}|f|^2 e^{-\varphi}
\end{eqnarray*}
provided $\varepsilon\le\varepsilon_S\ll 1$.
Set $g=\lambda_\varepsilon f-u_\varepsilon$. It is easy to see that $g$ is a desired function.
\end{proof}

\begin{lemma}\label{lm:keyLemma}
 Let $V\subset\subset U$ be two bounded pseudoconvex domains in ${\mathbb C}^n$. Suppose that $\{\varphi_j\}\subset PSH^-(U)$ satisfies condition $(\ast)$ and $\varphi_j$ converges almost everywhere on $U$ to a function $\varphi\in PSH^-(U)$. For every $f\in A^2(U,\varphi)$, there exists $f_j\in A^2(V,\varphi_j)$ such that
 $$
 {\lim\sup}_{j\rightarrow \infty} \|f_j\|_{L^2(V,\varphi_j)}\le \|f\|_{L^2(U,\varphi)}
 $$
  and $\|f_j-f\|_{L^2(V)}\rightarrow 0$.
\end{lemma}

\begin{proof}
 Let $E$ be the complete pluripolar set in condition $(\ast)$. It is known that there is a function
 $
 \varrho\in PSH^{-}(\overline{V})\cap C^\infty(\overline{V}\backslash E)
 $
  such that $\varrho=-\infty$ on $E\cap \overline{V}$ (cf. \cite{DemaillyBook}, Chapter 3, Lemma 2.2). Replacing $\varrho$ by $\varrho-1$, we may assume that $\varrho<-1$ holds on $V$.
Set
$$
\psi=-\log \left(-\varrho\right).
$$
Let $\chi$ be as above. Set
$$
\lambda_\varepsilon=\chi(\log(-\psi)+\log\varepsilon),\ \ \ 0<\varepsilon\ll1.
$$
 Applying Theorem \ref{th:Berndtsson-Donnelly-Fefferman} with $\psi$ and $\varphi$ replaced by $\psi/2$ and $\varphi_j+\psi/2$ respectively, we then obtain a solution $u_{j,\varepsilon}$ of  $\bar{\partial} u= f \bar{\partial}\lambda_\varepsilon$ on $V$ satisfying
\begin{eqnarray*}
\int_V |u_{j,\varepsilon}|^2 e^{-\varphi_j}
 & \le & C_0 \int_V |f|^2 |\bar{\partial}\lambda_{\varepsilon}|^2_{ i\partial\bar{\partial}\psi}e^{-\varphi_j}\\
& \le & C_0  \varepsilon^{2}\int_{S_\varepsilon}|f|^2 e^{-\varphi_j}
\end{eqnarray*}
where $S_\varepsilon:=\overline{V}\cap \{-\psi\le 1/\varepsilon\}$ and $C_0>0$ is a universal constant.

Since $e^{-\varphi_j}$ is bounded by a positive $L^1$ function $\phi_{S_\varepsilon}$ on $S_\varepsilon$ by condition $(\ast)$, and $f\in L^\infty(V)$, so we obtain
\begin{eqnarray*}
\int_{S_\varepsilon}|f|^2 e^{-\varphi_j}
& \rightarrow & \int_{S_\varepsilon}|f|^2 e^{-\varphi}
\end{eqnarray*}
 in view of the dominated convergence theorem. Set
$$
f_{j,\varepsilon}=\lambda_\varepsilon f-u_{j,\varepsilon}.
$$
We then have $f_{j,\varepsilon}\in {\mathcal O}(V)$ such that for every $j\ge j_\varepsilon\gg 1$,
$$
\|f_{j,\varepsilon}\|_{L^2(V,\varphi_j)}\le (1+C_0\varepsilon)\|f\|_{L^2(U,\varphi)},
$$
 and since $\varphi_j$ and $\varphi$ are non-positive,
$$
\|f_{j,\varepsilon}-f\|_{L^2(V)}^2\le 2 \int_{\{-\psi\ge \frac1{2\varepsilon}\}}|f|^2 +  C_0\varepsilon^2 \int_U |f|^2 e^{-\varphi}.
$$
 It suffices to take a subsequence from $\{f_{j,\varepsilon}\}$.
\end{proof}

\begin{proposition}\label{prop:On_Diagonal}
Under the conditions of Theorem \ref{th:weighted_Continue}, we have
 $$
K_j(z)\rightarrow K(z),\ \ \ z\in \Omega.
$$
\end{proposition}

\begin{proof}
 Let $S$ be a compact set in $\Omega$ and $w\in S$ be arbitrarily fixed. Set $f(z)=K(z,w)$ and $\Omega_\varepsilon=\{\rho<-\varepsilon\}$ for $\varepsilon>0$, where $\rho$ is a negative continuous psh exhaustion function of $\Omega$. By virtue of Lemma \ref{lm:keyLemma}, there exists $f_{j,\varepsilon}\in A^2(\Omega_\varepsilon,\varphi_j)$ such that
  $$
 {\lim\sup}_{j\rightarrow \infty} \|f_{j,\varepsilon}\|_{L^2(\Omega_\varepsilon,\varphi_j)}\le \|f\|_{L^2(\Omega,\varphi)}=\sqrt{K(w)}
 $$
  and $f_{j,\varepsilon}(w)\rightarrow f(w)$ as $j\rightarrow \infty$. On the other hand, Lemma \ref{lm:simple} yields a function $g_{j,\varepsilon}\in A^2(\Omega,\varphi_j)$ with $g_{j,\varepsilon}(w)=f_{j,\varepsilon}(w)$ and
  $$
\|g_{j,\varepsilon}\|_{L^2(\Omega,\varphi_j)}\le (1+{\rm const}_S/|\log \varepsilon|)\|f_{j,\varepsilon}\|_{L^2(\Omega_\varepsilon,\varphi_j)}.
$$
It follows that
$$
{\lim\inf}_{j\rightarrow \infty} K_j(w)\ge \frac{|g_{j,\varepsilon}(w)|^2}{\|g_{j,\varepsilon}\|^2_{L^2(\Omega,\varphi_j)}}\ge (1+{\rm const}_S/|\log \varepsilon|)^{-2} K(w).
$$
Since $\varepsilon$ can be arbitrarily small, so we get
$$
{\lim\inf}_{j\rightarrow \infty} K_j(w)\ge  K(w).
$$
Combining with Proposition \ref{prop:Upper_Semicontinuity}, we conclude the proof.
\end{proof}

\begin{proof}[Proof of Theorem \ref{th:weighted_Continue}]
 Set $\varphi_{j,k}=\max\{\varphi_j,-k\}$ and $\varphi_{0,k}=\max\{\varphi,-k\}$. Let
$
K_{j,k}(z,w)
$
denote the Bergman kernel with weight $\varphi_{j,k}$ on $\Omega$.
Since $\varphi_{j,k}\ge \varphi_j$, so 
$$
K_j(\cdot,w)\in L^2(\Omega,\varphi_j)\subset L^2(\Omega,\varphi_{j,k}),
$$
and we have
  \begin{eqnarray*}
   && \int_{\Omega}|K_j(\cdot,w)-K_{j,k}(\cdot,w)|^2 e^{-\varphi_{j,k}}\\
    & = &  \int_{\Omega}|K_j(\cdot,w)|^2e^{-\varphi_{j,k}}+\int_{\Omega}|K_{j,k}(\cdot,w)|^2e^{-\varphi_{j,k}}-2K_j(w)\\
   & \le & K_{j,k}(w)-K_j(w).
  \end{eqnarray*}
  
   Set $\Omega_\varepsilon=\{\rho<-\varepsilon\}$ for $\varepsilon\ll1$, where $\rho$ is a negative continuous psh exhaustion function on $\Omega$. We then have
  \begin{eqnarray*}
   && \| K_j(\cdot,w)-K(\cdot,w)\|_{L^2(\Omega_\varepsilon)}\\
    & \le &   \| K_j(\cdot,w)-K_{j,k}(\cdot,w)\|_{L^2(\Omega_\varepsilon)}+ \| K_{j,k}(\cdot,w)-K_{0,k}(\cdot,w)\|_{L^2(\Omega_\varepsilon)}\\
                                                    && +  \|K_{0,k}(\cdot,w)-K(\cdot,w)\|_{L^2(\Omega_\varepsilon)} \\
        & \le &         \| K_j(\cdot,w)-K_{j,k}(\cdot,w)\|_{L^2(\Omega,\varphi_{j,k})}+ \| K_{j,k}(\cdot,w)-K_{0,k}(\cdot,w)\|_{L^2(\Omega_\varepsilon)}\\
                                                    && +  \|K_{0,k}(\cdot,w)-K(\cdot,w)\|_{L^2(\Omega,\varphi_{0,k})} \\
      & \le &  (K_{j,k}(w)-K_j(w))^{1/2}+(K_{0,k}(w)-K(w))^{1/2}\\
         &&  +   \| K_{j,k}(\cdot,w)-K_{0,k}(\cdot,w)\|_{L^2(\Omega_\varepsilon)}.
  \end{eqnarray*}
 Let $K^\varepsilon_{0,k}$ denote the Bergman kernel with weight $\varphi_{0,k}$ on $\Omega_\varepsilon$. Notice that
  \begin{eqnarray*}
   && \| K_{j,k}(\cdot,w)-K_{0,k}^\varepsilon(\cdot,w)\|_{L^2(\Omega_\varepsilon)}^2\\
    & \le & \| K_{j,k}(\cdot,w)-K_{0,k}^\varepsilon(\cdot,w)\|_{L^2(\Omega_\varepsilon,\varphi_{0,k})}^2\\
     & = &  \int_{\Omega_\varepsilon}|K_{j,k}(\cdot,w)|^2e^{-\varphi_{0,k}}+\int_{\Omega_\varepsilon}|K_{0,k}^\varepsilon(\cdot,w)|^2e^{-\varphi_{0,k}}-2K_{j,k}(w)\\
   & = & \int_{\Omega_\varepsilon}|K_{j,k}(\cdot,w)|^2e^{-\varphi_{0,k}}+K_{0,k}^\varepsilon(w)-2K_{j,k}(w),
  \end{eqnarray*}
  and 
  $$
  \| K_{0,k}(\cdot,w)-K_{0,k}^\varepsilon(\cdot,w)\|_{L^2(\Omega_\varepsilon)}^2\le K_{0,k}^\varepsilon(w)-K_{0,k}(w).
  $$
  Since 
  $$
  |K_{j,k}(z,w)|^2\le K_{j,k}(z)K_{j,k}(w)\le K_\Omega(z)K_\Omega(w)
  $$
   where $K_\Omega$ is the (standard) Bergman kernel of $\Omega$, it follows from the dominated convergence theorem that
  $$
  \int_{\Omega_\varepsilon}|K_{j,k}(\cdot,w)|^2(e^{-\varphi_{0,k}}-e^{-\varphi_{j,k}})\rightarrow 0
  $$
  as $j\rightarrow \infty$. Thus for  every $0<\tau\ll 1$,
  $$
  \int_{\Omega_\varepsilon}|K_{j,k}(\cdot,w)|^2e^{-\varphi_{0,k}}\le \int_{\Omega}|K_{j,k}(\cdot,w)|^2 e^{-\varphi_{j,k}}+\tau=K_{j,k}(w)+\tau,
  $$
  provided $j\ge j(k,\varepsilon,\tau)\gg 1$. It follows that
  \begin{eqnarray}\label{eq:S3_4}
   &&  \| K_j(\cdot,w)-K(\cdot,w)\|_{L^2(\Omega_\varepsilon)}\nonumber\\
   & \le &   (K_{j,k}(w)-K_j(w))^{1/2}+(K_{0,k}(w)-K(w))^{1/2}\nonumber\\
   && +  (K_{0,k}^\varepsilon(w)-K_{j,k}(w)+\tau)^{1/2}+(K_{0,k}^\varepsilon(w)-K_{0,k}(w))^{1/2}.
  \end{eqnarray}
  By virtue of Proposition \ref{prop:On_Diagonal}, we have
  $$
  \lim_{j\rightarrow \infty} K_j(w)=K(w)\ \ \ {\rm and\ \ \ }
    \lim_{j\rightarrow \infty} K_{j,k}(w)=K_{0,k}(w).
  $$
   On the other hand, it is easy to verify that
  $$
  \lim_{\varepsilon\rightarrow 0} K^\varepsilon_{0,k}(w)=K_{0,k}(w) \ \ \ {\rm and\ \ \ } \lim_{k\rightarrow \infty} K_{0,k}(w)=K(w).
  $$
  Thus we get
   $$
  \lim_{j\rightarrow \infty} K_j(z,w)=K(z,w)
 $$
 in view of  (\ref{eq:S3_4}) and the mean value inequality.
 \end{proof}

\begin{remark}
By Cauchy's integrals, we may show  that
$$
\frac{\partial^{|\mu|+|\nu|} K_j(z,w)}{\partial z^\mu \partial \bar{w}^\nu} \rightarrow \frac{\partial^{|\mu|+|\nu|} K(z,w)}{\partial z^\mu \partial \bar{w}^\nu}
$$
for all multi-indices $\mu$ and $\nu$.
\end{remark}

  \section{Applications to singularity theory of psh functions}

  The following result improves a key semi-continuity result for complex singularity exponents (cf. \cite{DemaillyKollar}, Lemma 3.2; see also \cite{Varchenko}, \cite{PhongSturm}).

  \begin{proposition}\label{prop:semi-continuity}
   Let $U$ be a bounded pseudoconvex domain in ${\mathbb C}^n$. Suppose that $\{\varphi_j\}\subset PSH^-(U)$ satisfies condition $(\ast)$ and $\varphi_j$ converges almost everywhere on $U$ to a function $\varphi\in PSH^-(U)$ such that $e^{-\varphi}\in L^1(U)$. For every $V\subset \subset U$, there exists $j_0\in {\mathbb Z}^+$ such that
 $$
\int_V e^{-\varphi_j}\le {\rm const.}
 $$
 for all $j\ge j_0$.
  \end{proposition}

  \begin{proof}
  Choose a pseudoconvex domain $W$ satisfying $V\subset\subset W\subset\subset U$.  Applying Lemma \ref{lm:keyLemma} with $f=1$, we get a function $f_j\in {\mathcal O}(W)$ such that
  $$
  \int_W |f_j|^2 e^{-\varphi_j}\le {\rm const.}
  $$
  for $j\gg 1$, and $\|f_j-1\|_{L^2(W)}\rightarrow 0$. It follows that $|f_j|\ge 1/2$ on $V$ when $j\gg 1$, so that
  $$
  \int_V  e^{-\varphi_j}\le {\rm const.}
  $$
  \end{proof}

  Combining Proposition \ref{prop:semi-continuity} with Example (2) in \S\,1, we immediately obtain the following result due to Berndtsson \cite{BerndtssonOpenness} (originally conjectured by Demailly-Koll\'ar in \cite{DemaillyKollar}):

  \begin{corollary}[Openness Theorem]
  Let $U$ be a bounded pseudoconvex domain and $\varphi\in PSH^{-1}(U)$ with $\int_U e^{-\varphi}<\infty$. Let $V$ be a relatively compact domain in $U$. Then there exists $p>1$ such that $
  \int_{V} e^{-p\varphi}<\infty$.
  \end{corollary}

  \begin{remark}
  After Berndtsson's work \cite{BerndtssonOpenness}, Guan-Zhou \cite{GuanZhou} proved a\/ {\it strong} openness theorem that $\int_U |f|^2 e^{-\varphi}<\infty$ for a fixed holomorphic function $f$ implies $\int_V |f|^2 e^{-p\varphi}<\infty$ for some $p>1$.  It is unclear whether the method developed here still applies to this more general case.  We refer to \cite{FarveJonsson}, \cite{Heip}, \cite{JpnssonMustata} and \cite{Lempert} for related works on openness theorems.
  \end{remark}

 An equivalent statement of the openness theorem is that if $\varphi$ is a psh function in a neighborhood $U$ of\/ $0$ such that $c_0(\varphi)<\infty$, then $e^{-c_0(\varphi)\varphi}$ is not $L^1$ in any neighborhood of\/ $0$. Actually, we have the following more general conclusion:

 \begin{proposition}\label{prop:Openness}
  Let $\varphi$ be a psh function in a neighborhood $U$ of\/ $0$ such that $c_0(\varphi)<\infty$. Then $e^{-c_0(\varphi)\varphi}/|\varphi|^r$ is not $L^1$ in any neighborhood of\/ $0$ for every $0\le r<1$.
 \end{proposition}

 \begin{proof}
  Fix a number $c>c_0(\varphi)=:t_0$. Set
$$
E:=\left\{z\in U:e^{-c\varphi}\ {\rm is\ not\ }L^1\ {\rm in\ any\ neighborhood\ of\ }z\right\}.
$$
By virtue of Bombieri's theorem (cf. \cite{HormanderBook}, Corollary 4.4.6),  $E$ is an analytic subset in $U$. Clearly, $0\in E$. Shrinking $U$ if necessary, we find $f_1,\cdots,f_m\in {\mathcal O}(U)$ such that $E\cap U=\cap_j f^{-1}_j(0)$ and $\sum_j |f_j|^2<e^{-1}$ on $U$. Furthermore, we may assume that $\varphi<-1$ on $U$. Set
$$
\psi=-\log \left(-\log \sum |f_j|^2 \right)
$$
and
$$
\phi_{r,\tau}=-r\log (-\varphi)+\tau\psi
$$
where $0<r, \tau<1$.
Notice that
$$
i\partial\bar{\partial} \phi_{r,\tau} \ge r i\partial \log (-\varphi)\wedge \bar{\partial} \log(-\varphi)+\tau i\partial \psi\wedge \bar{\partial}\psi
$$
and
$$
\partial \phi_{r,\tau}=-r \partial \log (-\varphi)+\tau \partial \psi.
$$
It follows that
\begin{eqnarray*}
i\partial\phi_{r,\tau}\wedge \bar{\partial} \phi_{r,\tau}  & \le & r^2(1+\sqrt{\tau})i \partial \log (-\varphi)\wedge \bar{\partial} \log (-\varphi)\\
&& + (\tau^{3/2}+\tau^2)i \partial\psi\wedge \bar{\partial} \psi.
\end{eqnarray*}
If $r<r'<1$, then
$$
r' i\partial\bar{\partial} \phi_{r,\tau}\ge i\partial\phi_{r,\tau}\wedge \bar{\partial} \phi_{r,\tau}
$$
provided $\tau\ll (\frac{r'}r-1)^2$.
Let $\chi$ be as above. Set
$$
\lambda_\varepsilon=\chi(\log(-\psi)+\log\varepsilon),\ \ \ 0<\varepsilon\ll1.
$$

 Suppose on the contrary that there exists a (pseudoconvex) neighborhood $V$ of $0$ such that
 $$
 \int_V e^{-c_0(\varphi)\varphi}/|\varphi|^r<\infty.
 $$
   Applying Theorem \ref{th:Berndtsson-Donnelly-Fefferman} with $\psi$ and $\varphi$ replaced by $\phi_{r,\tau}$ and $t\varphi+\tau\psi$ respectively, we then obtain a solution $u_{t,\varepsilon}$ of  $\bar{\partial} u=  \bar{\partial}\lambda_\varepsilon$ on $V$ satisfying
\begin{eqnarray*}
&& \int_V |u_{t,\varepsilon}|^2 e^{-r\log(-\varphi)-t\varphi}\\
 & \le & {\rm const}_{r'} \int_V  |\bar{\partial}\lambda_{\varepsilon}|^2_{\tau i\partial\bar{\partial}\psi}e^{-r\log(-\varphi)-t\varphi}\\
& \le & {\rm const}_{r'} \varepsilon^{2}\int_{1/(2{\varepsilon})\le -\psi\le 1/\varepsilon} e^{-r\log(-\varphi)-t\varphi}.
\end{eqnarray*}
Since $e^{-c\varphi}$ is $L^1$ over $V\cap \{-\psi\le 1/\varepsilon\}$, so we obtain
\begin{eqnarray*}
\int_{V\cap \{ -\psi\le 1/\varepsilon\}}e^{-r\log(-\varphi)-t\varphi}
& \rightarrow & \int_{V\cap \{ -\psi\le 1/\varepsilon\} } e^{-r\log(-\varphi)-t_0\varphi}
\end{eqnarray*}
as $t\rightarrow t_0$, in view of the dominated convergence theorem. The function $f_{t,\varepsilon}:=\lambda_\varepsilon-u_{t,\varepsilon}$ is holomorphic in $V$ and satisfies
$$
\int_V |f_{t,\varepsilon}|^2 e^{-t \varphi}/|\varphi|^r\le {\rm const}_{r'} \int_V e^{-t_0\varphi}/|\varphi|^r
$$
and $\|f_{t,\varepsilon}-1\|_{L^2(V)}\rightarrow 0$ as $t\rightarrow t_0$ and $\varepsilon\rightarrow 0$. It follows that for certain smaller neighborhood $W$ of $0$ we have $|f_{t,\varepsilon}|\ge 1/2$ provided $\varepsilon\ll1$ and $|t-t_0|\ll 1$, so that $\int_W e^{-t\varphi}<\infty$ for some $t>t_0$, contradicts with the definition of $t_0=c_0(\varphi)$.
 \end{proof}

\begin{remark}
 Proposition \ref{prop:Openness} does not hold for $r>1$. An elementary example is given by $\varphi(z)=\log |z|$.  Yet it is still possible that the case $r=1$ is true. On the other hand, the example $\varphi(z)=\log |z|-(-\log|z|)^{1/2}$, where $|z|<1$, shows that there does not exist in general a number $r>1$ such that $e^{-c_0(\varphi)\varphi}/|\varphi|^r$ is $L^1$ in some neighborhood of $0$.
\end{remark}

   Similar ideas also yield an openness theorem for $S^1-$invariant psh functions near infinity. Let ${\mathcal F}$ denote the set of positive continuous psh functions $\varphi$ on ${\mathbb C}^n$ satisfying $\varphi(z)\rightarrow \infty$ as $|z|\rightarrow \infty$. For every $\varphi\in {\mathcal F}$, we define the log canonical threshold $c_\infty(\varphi)$ of $\varphi$ at\/ $\infty$ as
   $$
   c_\infty(\varphi):=\inf\{t> 0: e^{-t\varphi}\ {\rm is\ } L^1\ {\rm in \ } {\mathbb C}^n\}.
   $$ For every $t\in {\mathbb R}^+$, we denote by $K_t$ the Bergman kernel with weight $t\varphi$ on ${\mathbb C}^n$. Set
   $$
   c'_\infty(\varphi):=\inf\{t> 0: K_t(0)\neq 0\}
   $$
   and
   $$
    c''_\infty(\varphi):=\inf\{t> 0: K_t\ {\rm is\ not\ identically\ } 0\}.
   $$
 Clearly, we have $c_\infty(\varphi)\ge c'_\infty(\varphi)\ge  c''_\infty(\varphi)$. On the other hand, the following elementary fact holds.

 \begin{lemma}\label{lm:circular}
  If $\varphi$ is $S^1-$invariant, i.e., $\varphi(e^{i\theta}z)=\varphi(z)$ for every $\theta\in {\mathbb R}$, then $c_\infty(\varphi)=c'_\infty(\varphi)$.
 \end{lemma}

 \begin{proof}
  Suppose $K_t(0)\neq 0$. Since $K_t(z,0)$ is an entire function on ${\mathbb C}^n$, so we have
  $$
  K_t(z,0)=\sum c_{\alpha_1\cdots\alpha_n} z_1^{\alpha_1}\cdots z_n^{\alpha_n}.
  $$
  Notice that $ z_1^{\alpha_1}\cdots z_n^{\alpha_n}\bot\,1$ in $L^2({\mathbb C}^n,t\varphi)$ whenever $\sum \alpha_j>0$, for $\varphi$ is $S^1-$invariant. It follows that
  $$
  K_t(0)=\int_{{\mathbb C}^n} |K_t(\cdot,0)|^2 e^{-t\varphi}\ge |c_0|^2 \int_{{\mathbb C}^n}  e^{-t\varphi}=K_t(0)^2 \int_{{\mathbb C}^n}  e^{-t\varphi}.
  $$
  Thus we have $e^{-t\varphi}\in L^1({\mathbb C}^n)$, so that  $c_\infty(\varphi)\le c'_\infty(\varphi)$.
 \end{proof}

\begin{proposition}\label{prop:OpennessInfinity}
 For every $\varphi\in {\mathcal F}$, we have
\begin{enumerate}
 \item $K_{c'_\infty(\varphi)\varphi}(0)=0$ and $K_{c''_\infty(\varphi)\varphi}\equiv 0$.
 \item If $\varphi$ is $S^1-$invariant, then $e^{-c_\infty(\varphi)\varphi}$ is not $L^1$ in ${\mathbb C}^n$.
 \end{enumerate}
\end{proposition}

\begin{proof}
 (1) follows directly from the following proposition. (2) follows from (1) and Lemma \ref{lm:circular}.
\end{proof}

   \begin{proposition}\label{prop:Continue_Infinity}
  If $\varphi\in {\mathcal F}$, then $K_t(z)$ is continuous in $t$ over ${\mathbb R}^+$.
 \end{proposition}

 \begin{proof}
   Let $t_0\in {\mathbb R}^+$. Set
   $$
   \psi=- \log (2/t_0+\varphi).
   $$
      We then have
   \begin{eqnarray*}
   i\partial\bar{\partial} (t\varphi+\psi) & = & \frac{2t/t_0-1+ t\varphi}{2/t_0+\varphi}i\partial\bar{\partial}\varphi+\frac{i\partial\varphi\wedge \bar{\partial}\varphi}{(2/t_0+\varphi)^2}\\
   & \ge &  i\partial\psi\wedge \bar{\partial}\psi
   \end{eqnarray*}
   provided $|t-t_0|\le t_0/2$.
   Let $\chi$ be as above. Set
   $$
   \lambda_\varepsilon=\chi(\log(-\psi)+\log \varepsilon),\ \ \ \varepsilon\ll1.
   $$
      Let $w\in B_R:=\{|z|<R\}$. Applying Theorem \ref{th:Berndtsson-Donnelly-Fefferman} with $\varphi$ and $\psi$ replaced by $t\varphi+\psi/2$ and $\psi/2$ respectively, we get a solution $u_t$ of
   $$
   \bar{\partial}u=K_{t_0}(\cdot,w)\bar{\partial}\lambda_\varepsilon
   $$
   such that
   \begin{eqnarray*}
    \int_{{\mathbb C}^n} |u_t|^2 e^{-t\varphi} & \le & C_0\int_{{\mathbb C}^n} |K_{t_0}(\cdot,w)|^2 |\bar{\partial}\lambda_\varepsilon|^2_{ i\partial\bar{\partial} (t\varphi+\psi)} e^{-t\varphi}\\
    & \le & C_0\int_{\frac1{2\varepsilon}\le-\psi\le \frac1{\varepsilon}} \frac{|K_{t_0}(\cdot,w)|^2}{\psi^2}  e^{-t\varphi}\\
    & \le & C_0 \varepsilon^2 \int_{\frac1{2\varepsilon}\le-\psi\le \frac1{\varepsilon}} {|K_{t_0}(\cdot,w)|^2}  e^{-t_0\varphi}\\
    & \le & C_0 \varepsilon^2 K_{t_0}(w)
       \end{eqnarray*}
       provided $|t-t_0|\le \eta_\varepsilon\ll1$. Here $C_0$ is a universal constant. If $\varepsilon\ll1$, then $B_{R+1}\subset \{-\psi<\frac1{2\varepsilon}\}$. Since $u_t$ is holomorphic on $\{-\psi<\frac1{2\varepsilon}\}$, so the mean value inequality yields
       \begin{eqnarray*}
        |u_t(w)|^2 & \le & {\rm const}_n \int_{B_{R+1}} |u_t|^2\\
        & \le & {\rm const}_{n,t_0,R} \int_{B_{R+1}} |u_t|^2 e^{-t\varphi}\\
        & \le & {\rm const}_{n,t_0,R}\, \varepsilon^2 K_{t_0}(w).
       \end{eqnarray*}
       It follows that $f_t:=\lambda_\varepsilon K_{t_0}(\cdot,w)-u_t$ is an entire function satisfying
       $$
       |f_t(w)|\ge K_{t_0}(w)-{\rm const}_{n,t_0,R}\,\varepsilon
       $$
       and
       $$
       \|f_t\|_{L^2({\mathbb C}^n,t\varphi)}\le (1+C_0\varepsilon)\sqrt{K_{t_0}(w)}
       $$
       provided $|t-t_0|\le \eta_\varepsilon\ll1$. Thus
       $$
       {\lim\inf}_{t\rightarrow t_0} K_t(w)\ge K_{t_0}(w).
       $$
       Interchanging the roles of $t$ and $t_0$, we obtain
       $$
       {\lim}_{t\rightarrow t_0} K_t(w)=K_{t_0}(w).
       $$
 \end{proof}

 \begin{problem}
  Is $e^{-c_\infty(\varphi)\varphi}\notin L^1({\mathbb C}^n)$ for every $\varphi\in {\mathcal F}$?
 \end{problem}

\section{Proof of Theorem \ref{th:Weighted_Holder}}
   It suffices to verify the following two propositions.

 \begin{proposition}\label{prop:Holder_Diagonal_1}
 If $w\in \Omega\backslash \{0\}$, then $K_t(w)$ is H\"older continuous of order $\beta$ at $t=0$ for every $\beta<\frac{c_0(\varphi_0)-1}{c_0(\varphi_0)+1}\alpha$.
 \end{proposition}

 \begin{proof}
  Set
$
f(z):=K_{0} (z,w)/\sqrt{K_{0}(w)}$.
Fix $1/2<\gamma<1$ for a moment. Let $\chi_\gamma:{\mathbb R}\rightarrow [0,1]$ be a smooth function satisfying $\chi_\gamma|_{(0,\infty)}=0$ and $\chi_\gamma|_{(-\infty,\log \gamma)}=1$. Set
$$
\lambda_{\gamma,\varepsilon}=\chi_\gamma(\log(-\varphi_{0})-\log(-\log\varepsilon)),\ \ \ 0<\varepsilon\ll1.
$$
Applying Theorem \ref{th:Berndtsson-Donnelly-Fefferman} with $\psi =-\frac12 \log (-\varphi_{0})$ and $\varphi$ replaced by $\varphi_t+2n\log |z-w|+\psi$, we find a solution $u_t$ of  $\bar{\partial} u= f \bar{\partial}\lambda_{\gamma,\varepsilon}$ on $\Omega$ satisfying
\begin{eqnarray*}
\int_{\Omega} |u_t|^2 e^{-\varphi_t-2n\log |z-w|} & \le & 24 \int_{\Omega} |f|^2 |\bar{\partial}\lambda_{\gamma,\varepsilon}|^2_{i\partial\bar{\partial}\psi}e^{-\varphi_t-2n\log |z-w|}\\
& \le & \frac{C}{\delta_{\gamma,\varepsilon}(w)^{2n}}\int_{A_{1,\varepsilon}\backslash A_{\gamma,\varepsilon}}{|f|^2}e^{-\varphi_t}.
\end{eqnarray*}
provided $\varepsilon\ll 1$, where $C>0$ is a generic constant independent of $t,\varepsilon,w$,
$$
A_{s,\varepsilon}=\left\{ -\varphi_{0}\le -s\log \varepsilon\right\},\ \ \ s>0,
$$
and $\delta_{\gamma,\varepsilon}(w)=d(w,A_{1,\varepsilon}\backslash A_{\gamma,\varepsilon})$.
 Since
$$
|e^{\varphi_t(z)}-e^{\varphi_{0}(z)}|\le C |t|^\alpha,\ \ \ z\in \Omega,
$$
so we have
$$
 e^{\varphi_{0}-\varphi_t}  \le  1+C|t|^\alpha e^{-\varphi_t}\le 1+\frac{C|t|^\alpha}{e^{\varphi_{0}}-C|t|^\alpha}=\frac{1}{1-C|t|^\alpha e^{-\varphi_{0}}}\le  \frac{1}{1-C|t|^\alpha/\varepsilon},
$$
on $A_{1,\varepsilon}$ provided $|t|^\alpha/\varepsilon\ll 1$,
and
$$
e^{\varphi_t-\varphi_0}\le 1+C|t|^\alpha e^{-\varphi_0}\le 1+C|t|^\alpha/\varepsilon.
$$
Notice that
$$
|f(z)|\le \sqrt{K_0(z)}\le C
$$
on $A_{1,\varepsilon}\backslash A_{\gamma,\varepsilon}$ provided $\varepsilon \ll1$, for $e^{\varphi_0}$ is a continuous function with an isolated zero at $0$. Since $e^{-\varphi_0}$ is $L^1$ in a neighborhood $U$ of $0$, so the volume $|A_{1,\varepsilon}\backslash A_{\gamma,\varepsilon}|$ of $A_{1,\varepsilon}\backslash A_{\gamma,\varepsilon}$ satisfies
$$
|A_{1,\varepsilon}\backslash A_{\gamma,\varepsilon}|\le   \varepsilon^{c\gamma}\int_U e^{-c\varphi_0}
$$
for every $1<c<c_0(\varphi_0)$, and
$$
\int_{A_{1,\varepsilon}\backslash A_{\gamma,\varepsilon}}{|f|^2}e^{-\varphi_0}\le C |A_{1,\varepsilon}\backslash A_{\gamma,\varepsilon}|/\varepsilon\le {\rm const}_c\, \varepsilon^{c\gamma-1}.
$$
It follows that
\begin{eqnarray*}
 \int_{\Omega} |u_t|^2 e^{-\varphi_t-2n\log |z-w|} & \le & \frac{{\rm const}_c}{\delta_{\gamma,\varepsilon}(w)^{2n}}\cdot \frac{\varepsilon^{c\gamma-1}}{1-C|t|^\alpha/\varepsilon}.
\end{eqnarray*}
Set $f_t=\lambda_{\gamma,\varepsilon} f-u_t$. Then $f_t$ is holomorphic on $\Omega$ with $f_t(w)=f(w)=\sqrt{K_{0}(w)}$ and
\begin{eqnarray*}
\|f_t\|_{L^2(\Omega,\varphi_t)} & = & \|\lambda_{\gamma,\varepsilon}f\|_{L^2(\Omega,\varphi_t)}+\|u_t\|_{L^2(\Omega,\varphi_t)}\\
& \le & \frac{1}{(1-C|t|^\alpha/\varepsilon)^{1/2}}\left(1+\frac{{\rm const}_c}{\delta_{\gamma,\varepsilon}(w)^{n}}\varepsilon^{\frac{c\gamma-1}2}\right)\\
& =: & a_{t,\varepsilon}(w),
\end{eqnarray*}
so that
$$
K_t(w)\ge K_{0}(w)/a_{t,\varepsilon}(w)^{2}.
$$
Next we set
$$
 g_t(z):=K_{t} (z,w)/\sqrt{K_{t}(w)},\ \ \ z\in \Omega.
$$
Similar as above, we have a solution $u_0$ of $\bar{\partial} u= g_t \bar{\partial}\lambda_{\gamma,\varepsilon}$ on $\Omega$ satisfying
\begin{eqnarray*}
\int_{\Omega} |u_0|^2 e^{-\varphi_0-2n\log |z-w|} & \le & 24 \int_{\Omega} |g_t|^2 |\bar{\partial}\lambda_{\gamma,\varepsilon}|^2_{i\partial\bar{\partial}\psi}e^{-\varphi_0-2n\log |z-w|}\\
& \le & \frac{C}{\delta_{\gamma,\varepsilon}(w)^{2n}}\int_{A_{1,\varepsilon}\backslash A_{\gamma,\varepsilon}}{|g_t|^2}e^{-\varphi_0}\\
& \le & \frac{C\varepsilon^{c\gamma-1}}{\delta_{\gamma,\varepsilon}(w)^{2n}}
\end{eqnarray*}
for $|g_t(z)|\le \sqrt{K_t(z)}\le C$ provided $\varepsilon\ll1$. Clearly, $g_0:=\lambda_{\gamma,\varepsilon} g_t-u_0$ is holomorphic on $\Omega$ such that $g_0(w)=\sqrt{K_t(w)}$ and
\begin{eqnarray*}
\|g_0\|_{L^2(\Omega,\varphi_0)} & = & \|\lambda_{\gamma,\varepsilon}g_t\|_{L^2(\Omega,\varphi_0)}+\|u_0\|_{L^2(\Omega,\varphi_0)}\\
& \le & (1+C|t|^\alpha/\varepsilon)^{1/2}\left(1+\frac{C\varepsilon^{\frac{c\gamma-1}2}}{\delta_{\gamma,\varepsilon}(w)^{n}}\right)\\
& =: & b_{t,\varepsilon}(w),
\end{eqnarray*}
so that
$$
K_0(w)\ge  K_{t}(w)/b_{t,\varepsilon}(w)^{2}.
$$
Notice that
\begin{eqnarray*}
a_{t,\varepsilon}(w) & = & 1+O\left(|t|^\alpha/\varepsilon +\varepsilon^{\frac{c\gamma-1}2}/\delta_{\gamma,\varepsilon}(w)^n\right)\\
b_{t,\varepsilon}(w) & = & 1+O\left(|t|^\alpha/\varepsilon +\varepsilon^{\frac{c\gamma-1}2}/\delta_{\gamma,\varepsilon}(w)^n\right)
\end{eqnarray*}
provided $|t|^\alpha/\varepsilon +\varepsilon^{\frac{c\gamma-1}2}/\delta_{\gamma,\varepsilon}(w)^n\ll 1$. Thus
 \begin{equation}\label{eq:Holder}
  |K_t(w)-K_0(w)|\le C\left(|t|^\alpha/\varepsilon +\varepsilon^{\frac{c\gamma-1}2}/\delta_{\gamma,\varepsilon}(w)^n\right).
 \end{equation}
 If $\varepsilon=|t|^{\frac{2\alpha}{c\gamma+1}}\ll 1$, then $\delta_{\gamma,\varepsilon}(w)\ge {\rm const}_w>0$, so that $K_t(w)$ is H\"older continuous of order $\frac{c\gamma-1}{c\gamma+1}\alpha$ at $t=0$. Since $c $ and $\gamma$ can be arbitrarily close to $c_0(\varphi_0)$ and $1$ respectively,  we conclude the proof.
\end{proof}

 \begin{proposition}\label{prop:Holder_Diagonal_2}
 $K_t(0)$ is H\"older continuous of order $\beta$ at $t=0$ for every
   $
   \beta<\frac{\eta_0\alpha}{1+\eta_0\tau_0},
   $
   where
   $$
    \eta_0=\min\left\{\frac{1}{\mathcal L_0(\varphi)},\frac{c_0(\varphi_0)-1}{2n}\right\},\ \ \ \tau_0=\min\left\{\frac{c_0(\varphi_0)-1}{2\eta_0}-n,1\right\}.
   $$
 \end{proposition}

 \begin{proof}
 Let $\nu>\mathcal L_0(\varphi_0)$. We then have
 $$
 e^{\varphi_0(z)}\ge {\rm const}_\nu\, |z|^\nu
 $$
 on $A_{1,\varepsilon}\backslash A_{\gamma,\varepsilon}$, provided $\varepsilon\ll1$. Thus
 $$
 A_{1,\varepsilon}\backslash A_{\gamma,\varepsilon}\subset \{z:|z|\le {\rm const}_\nu\,\varepsilon^{\gamma/\nu}\},
 $$
 so that
 $$
 \delta_{\gamma,\varepsilon}(w)\ge |w|/2
 $$
 provided $\varepsilon^{\gamma/\nu}/|w|\ll 1$.
 Set
 $$
 \eta=\min\left\{\frac{\gamma}{\nu},\frac{c\gamma-1}{2n}\right\},\ \ \ \tau=\min\left\{\frac{c\gamma-1}{2\eta}-n,1\right\}.
 $$
  If $\varepsilon=|w|^{1/\eta}/C$ with $C\gg 1$, we then have
 $$
|K_t(w)-K_0(w)|\le C\left(\frac{|t|^\alpha}{|w|^{1/\eta}}+|w|^{\frac{c\gamma-1}{2\eta}-n}\right)
$$
in view of (\ref{eq:Holder}), provided $|t|^\alpha/|w|^{1/\eta}\ll 1$.
On the other hand, we claim that
$$
|K_t(w)-K_t(0)|\le C|w|.
$$
To see this, notice first that
$
K_t(z)\le K_\Omega(z)\le C
$
for all $z$ in a small neighborhood $U$ of $0$, where $K_\Omega$ is the (standard) Bergman kernel on $\Omega$.
Since
$$
\int_\Omega |K_t(\cdot,z)|^2 \le \int_\Omega |K_t(\cdot,z)|^2 e^{-\varphi_t}=K_t(z)\le C
$$
for all $z\in U$, it follows from Cauchy's integrals that for $w,w'$ sufficiently close to $0$,
$$
|K_t(w',w)-K_t(w)|\le C|w-w'|
$$
$$
|K_t(w,w')-K_t(w')|\le C|w-w'|,
$$
so that
$$
|K_t(w)-K_t(w')|\le C |w-w'|.
$$

Thus
\begin{eqnarray*}
 |K_t(0)-K_0(0)| & \le & C\left(\frac{|t|^\alpha}{|w|^{1/\eta}}+|w|^{\frac{c\gamma-1}{2\eta}-n}+|w|\right)\\
                 & \le &  C\left(\frac{|t|^\alpha}{|w|^{1/\eta}}+|w|^{\tau}\right)\\
                 & \le & C |t|^{\frac{\eta\tau\alpha}{1+\eta\tau}}
\end{eqnarray*}
provided $|w|=|t|^{\frac{\eta\alpha}{1+\eta\tau}}$. Since $c,\nu$ and $\gamma$ can be arbitrarily close to $c_0(\varphi_0),\mathcal L_0(\varphi_0)$ and $1$ respectively, so we conclude the proof.
\end{proof}

\begin{problem}
How to get the H\"older continuity of $K_t$ in $t$ when $c_0(\varphi_0)\le 1$?
\end{problem}

\section{Proof of Theorem \ref{th:principle}}

\begin{proposition}\label{prop:BergmanIntegral}
 Let $\Omega\subset {\mathbb C}^n$ be a bounded pseudoconvex domain. Let $\rho$ be a negative continuous psh  function on $\Omega$. Set
  $$
 \Omega^\varepsilon=\{z\in \Omega:-\rho(z)>\varepsilon\},\ \ \ \varepsilon>0.
 $$
  Let $S$ be a compact set in $\Omega$. Suppose
  $$
  S':=\{z\in \Omega:d(z,S)\le d(S,\partial \Omega)/2\}\subset \Omega^{\varepsilon_0}
  $$
  for some $\varepsilon_0>0$.  Let $K_{\Omega}$ denote the Bergman kernel on $\Omega$.
   Then for every  $0<r<1$,
  $$
  \int_{-\rho\le\varepsilon} |K_{\Omega}(\cdot,w)|^2 \le {\rm const}_{n,r}\, d(S,\partial\Omega)^{-2n} (\varepsilon/a)^{r}
  $$
  for all $w\in S$ and $\varepsilon\le \varepsilon_r\ll \varepsilon_0$.
  Here  $a=\inf_{S'}(-\rho)$.
\end{proposition}

\begin{proof}
 Let $\kappa:{\mathbb R}\rightarrow [0,1]$ be a smooth cut-off function such that $\kappa|_{(-\infty,1]}=1$, $\kappa|_{[3/2,\infty)}=0$ and $|\kappa'|\le 2$. We then have
$$
     \int_{ -\rho\le\varepsilon} |K_{\Omega}(\cdot,w)|^2 \le \int_\Omega \kappa(-\rho/\varepsilon) |K_{\Omega}(\cdot,w)|^2.
$$
 By the well-known property of the Bergman projection, we obtain
$$
  \int_\Omega \kappa(-\rho/\varepsilon) K_{\Omega}(\cdot,w)\cdot \overline{K_{\Omega}(\cdot,\zeta)}  = \kappa(-\rho(\zeta)/\varepsilon) K_{\Omega}(\zeta,w)-u(\zeta), \ \ \ \zeta\in\Omega,
$$
where $u$ is the $L^2(\Omega)-$minimal solution of the equation
$$
\bar{\partial} u= \bar{\partial}(\kappa(-\rho/\varepsilon) K_{\Omega}(\cdot,w))=:v.
$$
Since $\kappa(-\rho(w)/\varepsilon)=0$ provided $\frac32 \varepsilon\le \varepsilon_0$, so we have
\begin{equation}\label{eq:BergmanUpper}
 \int_{ -\rho\le\varepsilon} |K_{\Omega}(\cdot,w)|^2\le -u(w).
\end{equation}
  Set
            $$
      \psi=- r\log (-\rho),\ \ \ 0<r<1.
      $$
      Clearly, $\psi$ is psh and satisfies $ri\partial\bar{\partial}\psi\ge i\partial \psi\wedge \bar{\partial}\psi$, so that
             $$
 i \bar{v}\wedge v\le C_0 r^{-1} |\kappa'(-\rho/\varepsilon)|^2 |K_{\Omega}(\cdot,w)|^2 i\partial\bar{\partial}\psi
 $$
 for some numerical constant $C_0>0$. Thus by (\ref{eq:L2Minimal}) we obtain
       \begin{eqnarray*}
   \int_\Omega |u|^2e^{-\psi}
   & \le &  {\rm const}_{r} \int_{\varepsilon\le -\rho\le \frac32\varepsilon} |K_{\Omega}(\cdot,w) |^2 e^{-\psi}\\
   & \le & {\rm const}_{r}\, \varepsilon^{r}\int_{ -\rho\le \frac32\varepsilon} |K_{\Omega}(\cdot,w) |^2.
      \end{eqnarray*}
                   Since $e^{-\psi}\ge a^{r}$ on $S'$ and $u$ is holomorphic there, it follows from the mean value inequality that
    \begin{eqnarray*}
   |u(w)|^2 & \le & {\rm const}_n\, d(S,\partial \Omega)^{-2n} \int_{S'}|u|^2\\
   & \le & {\rm const}_n\, d(S,\partial \Omega)^{-2n} a^{-r} \int_{\Omega}|u|^2 e^{-\psi}\\
   & \le & {\rm const}_{n,r}\, d(S,\partial \Omega)^{-2n} (\varepsilon/a)^{r}\int_{ -\rho\le \frac32\varepsilon} |K_{\Omega}(\cdot,w) |^2.
   \end{eqnarray*}
   Thus by (\ref{eq:BergmanUpper}), we obtain
   \begin{eqnarray*}
   \int_{ -\rho\le\varepsilon} |K_{\Omega}(\cdot,w)|^2 \le  {\rm const}_{n,r}\, d(S,\partial \Omega)^{-n} (\varepsilon/a)^{r/2}\left(\int_{ -\rho\le \frac32\varepsilon} |K_{\Omega}(\cdot,w) |^2 \right)^{1/2}.
   \end{eqnarray*}
   Notice that
   $$
   \int_{-\rho\le \frac32\varepsilon} |K_{\Omega}(\cdot,w) |^2 \le \int_\Omega |K_{\Omega}(\cdot,w) |^2 =K_{\Omega}(w)\le {\rm const}_n\, d(S,\partial\Omega)^{-2n}
   $$
   provided $\frac32\varepsilon\le \varepsilon_0$.
   Thus
   $$
   \int_{ -\rho\le\varepsilon} |K_{\Omega,\varphi}(\cdot,w)|^2 \le {\rm const}_{n,r}\, d(S,\partial \Omega)^{-2n}  (\varepsilon/a)^{r/2}.
   $$
   Replacing $\varepsilon$ by $\frac32\varepsilon$ in the argument above, we obtain
  \begin{eqnarray*}
   \int_{ -\rho\le \frac32\varepsilon} |K_{\Omega}(\cdot,w)|^2  & \le & {\rm const}_{n,r}\, d(S,\partial \Omega)^{-2n} (3/2)^{r/2} (\varepsilon/a)^{r/2}
   \end{eqnarray*}
   provided $(3/2)^2\varepsilon\le \varepsilon_0$.
  Thus we may improve the upper bound by
  $$
   \int_{ -\rho\le\varepsilon} |K_{\Omega}(\cdot,w)|^2 \le {\rm const}_{n,r}\, d(S,\partial \Omega)^{-2n} (\varepsilon/a)^{r/2+r/4}.
   $$
   By induction, we conclude that for every $k\in {\mathbb Z}^+$,
  $$
   \int_{ -\rho\le\varepsilon} |K_{\Omega}(\cdot,w)|^2 \le {\rm const}_{n,r,k}\, d(S,\partial \Omega)^{-2n} (\varepsilon/a)^{r/2+r/4+\cdots+r/2^k}
   $$
   provided $(3/2)^k\varepsilon\le \varepsilon_0$. Since $r/2+r/4+\cdots+r/2^k\rightarrow 1$ as $k\rightarrow \infty$ and $r\rightarrow 1$, we conclude the proof.
\end{proof}

\begin{proof}[Proof of Theorem \ref{th:principle}]
Fix a pair $t\neq t_0$ for a moment.  Set $
\varepsilon=|t-t_0|^\alpha.
$
Since $\Omega_t$ is $\rho_t-$H\"older continuous of order $\alpha$, so there exist positive numbers $\gamma_3\gg \gamma_2\gg\gamma_1\gg 1$ and $\eta>0$ such that
\begin{equation}\label{eq:inclusion1}
 \{-\rho_t>\gamma_2\, \varepsilon\}\subset \{-\rho_{t_0}> \gamma_1\, \varepsilon\}=:\Omega_{t_0}'\subset  \{-\rho_t>\varepsilon\}
\end{equation}
and
\begin{equation}\label{eq:inclusion2}
 \{-\rho_t> (3/2)\gamma_2\,\varepsilon\}\supset \{-\rho_{t_0}> \gamma_3\, \varepsilon\}=:\Omega_{t_0}''
\end{equation}
provided $|t-t_0|\le \eta$. Let $\kappa$ be as above. Set
$$
\lambda_{t,\varepsilon}=1-\kappa(-\rho_t/(\gamma_2\,\varepsilon)).
$$
Let $S$ be a compact subset of the total set
 $$
 \Omega=\{(z,\tau):z\in \Omega_\tau,\tau\in (-1,1)\}.
 $$
 Let
  $
  S_{\tau}=\{z:(z,\tau)\in S\}.
  $
  Without loss of generality, we assume that $S_{t_0}\neq \emptyset$. Thus there exists a sufficiently small number $r_0$ (depending only on $S$) such that
 \begin{equation}\label{eq:inclusion3}
 \{(z,\tau):z\in S_\tau,|\tau-t_0|<r_0\} \subset \Omega_{t_0}'''\times (t_0-r_0,t_0+r_0)
 \end{equation}
 where
 $$
 \Omega_{t_0}'''=\{-\rho_{t_0}> 2\gamma_3\, \varepsilon\},
 $$
 provided $\varepsilon\ll 1$.
  Let $K'_{t_0}$ denote the Bergman kernel on $\Omega_{t_0}'$.
     Fix $z,w\in S_{t_0}$ for a moment.
   By the reproducing property, we have
   \begin{eqnarray}\label{eq:1}
    K_t(z,w) & = & \int_{\zeta\in \Omega_{t_0}'}K_t(\zeta,w)K_{t_0}'(z,\zeta)\nonumber\\
             & = & \int_{\zeta\in \Omega_{t_0}'} \lambda_{t,\varepsilon}(\zeta) K_t(\zeta,w)K_{t_0}'(z,\zeta)\nonumber\\
             && +\int_{\zeta\in \Omega_{t_0}'}(1-\lambda_{t,\varepsilon}(\zeta))K_t(\zeta,w)K_{t_0}'(z,\zeta)\nonumber\\
             & =: & I+II.
   \end{eqnarray}
  Since $\lambda_{t,\varepsilon}(w)=1$ and $\lambda_{t,\varepsilon}=0$ outside $\Omega_{t_0}'$ in view of (\ref{eq:inclusion1})--(\ref{eq:inclusion3}), so
  \begin{equation}\label{eq:2}
  I=\int_{\zeta\in \Omega_{t}} \lambda_{t,\varepsilon}(\zeta)K_{t_0}'(z,\zeta) K_t(\zeta,w)=K_{t_0}'(z,w)-\overline{u_t(w)},
  \end{equation}
  where $u_t$ is the $L^2(\Omega_t)-$minimal solution of
  $$
  \bar{\partial}u=\bar{\partial}(\lambda_{t,\varepsilon} K_{t_0}'(\cdot,z))=:v_t.
  $$
     Applying (\ref{eq:L2Minimal}) with $\psi=-r\log(-\rho_t)$ ($0<r<1$) and $\varphi=0$, we obtain
\begin{eqnarray*}
 \int_{\Omega_{t}} |{u}_t|^2 e^{-\psi} & \le & {\rm const}_r\,\int_{\Omega_{t}} |\kappa'(-\rho_t/\varepsilon)|^2 |K_{t_0}'(\cdot,z)|^2 e^{-\psi}\\
 & \le & {\rm const}_r\,\varepsilon^{r} \int_{\gamma_2\,{\varepsilon}<-\rho_t < \frac32\gamma_2\,{\varepsilon}}|K_{t_0}'(\cdot,z)|^2 \\
 & \le &  {\rm const}_r\,\varepsilon^{r}\int_{ -\rho_{t_0}\le \gamma_3\,{\varepsilon}}|K_{t_0}'(\cdot,z)|^2 \\
 & \le & {\rm const}_{r,S}\,\varepsilon^{2r}
\end{eqnarray*}
in view of Proposition \ref{prop:BergmanIntegral}. By the mean value inequality, we obtain
  \begin{equation}\label{eq:3}
 |{u}_t(w)|\le {\rm const}_{r,S}\, \varepsilon^r.
 \end{equation}
 On the other hand, we have
 \begin{eqnarray}\label{eq:4}
  II & \le & \int_{\Omega_{t_0}'\cap\{\lambda_{t,\varepsilon}\neq 1\}}|K_t(\cdot,w)K_{t_0}'(z,\cdot)| \le \int_{\Omega_{t_0}'\cap \{-\rho_t\le (3/2)\gamma_2\, \varepsilon\}}|K_t(\cdot,w)K_{t_0}'(z,\cdot)| \nonumber\\
  & \le & \left(\int_{-\rho_t\le (3/2)\gamma_2\, \varepsilon}|K_t(\cdot,w)|^2 \right)^{1/2}  \left(\int_{-\rho_{t_0}\le \gamma_3\,\varepsilon} |K_{t_0}'(z,\cdot)|^2 \right)^{1/2}\nonumber\\
  & \le & {\rm const}_{r,S}\,\varepsilon^{r}
 \end{eqnarray}
in view of Proposition \ref{prop:BergmanIntegral}. By (\ref{eq:1})--(\ref{eq:4}), we get
\begin{equation}\label{eq:5}
|K_t(z,w)-K_{t_0}'(z,w)|\le {\rm const}_{r,S}\,|t-t_0|^{r\alpha}.
\end{equation}
The point is that the constant of the RHS of (\ref{eq:5}) is independent of $t_0$. Thus for any pair $t\neq s$ with $|t-s|\le \eta\ll1$ we may take $t_0=\frac{t+s}2$ so that (\ref{eq:5}) holds for $t$ and $s$. By the triangle inequality, we finally get
$$
|K_t(z,w)-K_{s}(z,w)|\le {\rm const}_{r,S}\,|t-s|^{r\alpha}.
$$
  \end{proof}

     \begin{proposition}\label{prop:Holder}
       Let $\left\{\Omega_t:-1<t<1\right\}$ be a $C^2$ family of bounded pseudoconvex domains in ${\mathbb C}^n$ with $C^2$ boundaries. Then there exists a number $0<\alpha\le 1$ such that $K_t(z,w)$ is H\"older continuous of order $\alpha$ in $t$.
       \end{proposition}

  Proposition \ref{prop:Holder} follows directly from Theorem \ref{th:principle} and the following result essentially due to Diederich-Fornaess \cite{DiederichFornaess77}:

\begin{lemma}\label{lm:DiederichFornaess}
  Let $\left\{\Omega_t:-1<t<1\right\}$ be a $C^2$ family of bounded pseudoconvex domains in ${\mathbb C}^n$ with $C^2$ boundaries. For every $t_0\in (-1,1)$, there exist a compact set $S\subset \Omega_{{t_0}}$, an open neighborhood $I_0$ of $t_0$ and constants $K>0$, $0<\eta<1$ such that $S \times I_0$ is contained in the total set $ \Omega$ and
$$
\rho_t:=-(\delta_t e^{-K|z|^2})^\eta
$$
is psh on $\Omega_t\backslash S\times \{t\}$, $t\in I_0$. Here $\delta_t$ denotes the boundary distance of $\Omega_t$.
\end{lemma}

\begin{proof}
 For the sake of completeness, we will include a proof here. By virtue of Oka's lemma, we have $-\log \delta_t\in PSH(\Omega_t)$, so that
\begin{equation}\label{eq:inequality}
-i\partial\bar{\partial}\delta_t\ge -\frac{i\partial \delta_t\wedge \bar{\partial}\delta_t}{\delta_t}.
\end{equation}
For any point $z\in \Omega_t$ sufficiently close to $\partial \Omega_t$ (which is uniform in $t$ in a sufficiently small open neighborhood $I_0$ of $t_0$), we denote by $\hat{z}_t$ the projection of $z$ on $\partial \Omega_t$. Given $\zeta\in {\mathbb C}^n$, we have the following decomposition
$$
\zeta=\zeta' \oplus \zeta''
$$
where $\langle \partial \delta_t,\zeta' \rangle|_{\hat{z}_t}=0$. By (\ref{eq:inequality}), we have
\begin{eqnarray*}
-i\partial\bar{\partial}\delta_t(z;\zeta') & \ge & -\frac{|\langle \partial \delta_t(z),\zeta' \rangle|^2}{\delta_t(z)}=-\frac{|\langle (\partial \delta_t(z)-\partial \delta_t(\hat{z}_t)),\zeta' \rangle|^2}{\delta_t(z)}\\
& \ge & -C\delta_t(z)|\zeta|^2
\end{eqnarray*}
where $C>0$ is a generic independent of $t$. Since
$$
|\zeta''|=|\langle \partial \delta_t(\hat{z}_t),\zeta\rangle|\le |\langle \partial \delta_t(z),\zeta\rangle|+C\delta_t(z)|\zeta|,
$$
so
$$
-i\partial \bar{\partial}\delta_t (z;\zeta)\ge -C \delta_t(z)|\zeta|^2-C|\zeta|\,|\langle \partial \delta_t(z),\zeta\rangle|.
$$
Set $\psi_t=-\log \delta_t+K|z|^2$, $K>0$. Then
\begin{eqnarray*}
i\partial\bar{\partial}\psi_t(z;\zeta) & = & -\frac{i\partial\bar{\partial}\delta_t(z;\zeta)}{\delta_t(z)}+\frac{|\langle \partial\delta_t(z),\zeta \rangle|^2}{\delta_t(z)^2}+K|\zeta|^2\\
& \ge &(K-C)|\zeta|^2-C\frac{|\zeta||\langle \partial\delta_t(z),\zeta \rangle|}{\delta_t(z)}+\frac{|\langle \partial\delta_t(z),\zeta \rangle|^2}{\delta_t(z)^2}\\
& \ge &\frac12 \left(|\zeta|^2+\frac{|\langle \partial\delta_t(z),\zeta \rangle|^2}{\delta_t(z)^2}\right)
\end{eqnarray*}
provided $K$ sufficiently large. Since $\partial \psi_t=-\partial \delta_t/\delta_t+K\partial |z|^2$, we conclude that there is a number $0<\eta<1$ (independent of $t$) such that
$$
i\partial \bar{\partial}\psi_t\ge \eta i\partial \psi_t\wedge \bar{\partial} \psi_t.
$$
It suffices to take $\rho_t=-\exp(-\eta\psi_t)$.
\end{proof}

\begin{remark}
     It is possible to weaken the boundary regularity in Proposition \ref{prop:Holder} to Lipschitz continuity by \cite{HarringtonPSH}.
 \end{remark}

  We conclude this section by proposing the following

  \begin{problem}
  Is $K_t(z,w)$ H\"older continuous of order $\alpha$ in $t$ under the conditions of Theorem \ref{th:principle}?
  \end{problem}

 As we will see in the next section, the answer is positive when $n=1$.

\section{One dimensional case}

 The purpose of this section is to show the following

 \begin{theorem}\label{th:simply-connected}
        Let $\left\{\Omega_t:-1<t<1\right\}$ be a uniformly bounded family of simply-connected domains in ${\mathbb C}$. Let $\delta_t$ denote the Euclidean boundary distance of $\Omega_t$.
           Suppose $\Omega_t$ is $(-\delta_t)-$H\"older continuous of order $\alpha$ over $(-1,1)$.
         Then $K_t(z,w)$  is H\"older continuous of order $\alpha/2$ in $t$.
                \end{theorem}

                As a consequence, we obtain

              \begin{corollary}\label{cor:Riemann-mapping}
        Let $\{\Omega_t\}$ be as the theorem above. Suppose furthermore that $0\in \Omega_t$ for all $t$. Let $F_t:\Omega_t\rightarrow \Delta=\{z:|z|<1\}$ denote the Riemann mapping which satisfies $F_t(0)=0$ and $F_t'(0)>0$. Then $F_t(z)$ is H\"older continuous of order $\alpha/2$ in $t$.
       \end{corollary}

       \begin{proof}
   Since
   $$
   K_t(z,0)=F_t'(0)K_\Delta(F_t(z),0)F_t'(z)=\frac{F_t'(0)F_t'(z)}\pi
   $$
   and $F_t'(0)=\sqrt{\pi K_t(0)}$, it follows that
   $$
   F_t(z)=\frac{\sqrt{\pi}}{\sqrt{K_t(0)}} \int_0^z K_t(\cdot,0).
   $$
   The assertion follows immediately from Theorem \ref{th:simply-connected}.
  \end{proof}

       \begin{remark}
        It is a classical result of Carath\'eodory that if $\delta_t$ is continuous in $t$ then $F_t$ is also continuous in $t$ $($see \cite{TsujiBook}, Theorem IX. 13\,$)$.
       \end{remark}

       We begin with the following

\begin{proposition}\label{prop:bounded_exhaust}
 Let $\Omega$ be a bounded simply-connected domain in ${\mathbb C}$ and let $\delta$ denote the boundary distance of $\Omega$. Then there exists a continuous negative subharmonic function $\rho$ on $\Omega$ such that
 $$
 (\delta/r_\Omega)^2\le -\rho\le (\delta/r_\Omega)^{1/2}
 $$
 where $r_\Omega$ denotes the inradius of $\Omega$, i.e., the radius of the largest disc inscribed in $\Omega$.
\end{proposition}

\begin{proof}
Let $\Delta$ denote the unit disc. Let $\phi_0(z)$ denote the hyperbolic distance between $z\in \Delta$ and $0$, i.e.,
$
\phi_0(z)=\log \frac{1+|z|}{1-|z|}.
$
 Set $\psi(z)=-\frac1{1+|z|}$ for $z\in \Delta$. A straightforward calculation yields
$$
\frac{\partial^2\psi}{\partial z\partial \bar{z}}=\frac{1-|z|}{4|z|(1+|z|)^3}>0.
$$
It follows that
$
\rho_0:=-e^{-\phi_0}=1+2\psi
$
is subharmonic on $\Delta$.

Let $ds^2_{\rm hyp}=\lambda(z)|dz|^2$ denote the Poincar\'e hyperbolic metric of $\Omega$ and let $d_{\rm hyp}$ be the corresponding distance. Take a point $z_0\in \Omega$ such that $\delta(z_0)=r_\Omega$. Set $\phi=d_{\rm hyp}(z_0,\cdot)$. Let $F:\Omega\rightarrow \Delta$ be a conformal mapping such that $F(z_0)=0$. Since $\phi=\phi_0\circ F$, it follows that
$
\rho:=-e^{-\phi}
$
 is subharmonic on $\Omega$. Thanks to Koebe's $\frac14-$theorem, we have
   $$
   ds^2_{\rm hyp}\ge \frac{|dz|^2 }{4\delta^{2}}= \frac{|\nabla \delta|^2}{4\delta^{2}}|dz|^2\ \ \ {\rm a.e.}
   $$
    Thus
    $$
   \phi \ge \frac12 \log 1/\delta -\frac12 \log 1/r_\Omega,
    $$
   i.e., $-\rho\le (\delta/r_\Omega)^{1/2}$. To be more rigorous, we take a geodesic $\gamma$ with $\gamma(0)=z_0$, $\gamma(1)=z$ for an arbitrarily fixed point $z\in \Omega$  and a variation $\{\gamma_s: s\in (-\varepsilon,\varepsilon)\}$ of $\gamma$ inside $\Omega$ such that $\gamma_0=\gamma$, $\gamma_s(0)=z_0$ and $\gamma_s(1)=z$ for all $s$. There exists a sequence of numbers $s_j\rightarrow 0$ such that $\delta$ is differentiable a.e. along $\gamma_{s_j}$ for all $j$. Thus the hyperbolic length $|\gamma_{s_j}|_{\rm hyp}$ of $\gamma_{s_j}$ satisfies
    $$
    |\gamma_{s_j}|_{\rm hyp}\ge \frac12\left|\int_0^1 (\log \delta\circ\gamma_{s_j}(t))'dt \right|= \frac12 \log 1/\delta(z) -\frac12 \log 1/r_\Omega
    $$
     so that
   $$
   \phi(z)=\lim_{j\rightarrow \infty} |\gamma_{s_j}|\ge \frac12 \log 1/\delta(z) -\frac12 \log 1/r_\Omega.
   $$
      On the other side, it follows from the trivial estimate
   $$
   ds^2_{\rm hyp}\le \frac{4|dz|^2 }{\delta^{2}}= \frac{4|\nabla\delta|^2}{\delta^{2}}|dz|^2\ \ \ {\rm a.e.}
   $$
   that
   $$
    \phi \le 2\log 1/\delta -2\log 1/r_\Omega,
   $$
   i.e., $-\rho\ge (\delta/r_\Omega)^2$.
\end{proof}

Let $\{\Omega_t\}$ be as in Theorem \ref{th:simply-connected} and let $g_t$ denote the (negative) Green function of $\Omega_t$. We have the following H\"older continuity of $g_t$ in $t$:

\begin{proposition}\label{prop:GreenStability}
 Let $t_0\in (-1,1)$ and let $S_{t_0}$ be a compact set in $\Omega_{t_0}$. Then there exists a constant $C>0$ such that
 $$
 |g_t(z,w)-g_{t_0}(z,w)|\le C\,|t-t_0|^{\alpha/2}
 $$
 for all $z,w\in S_{t_0}$, provided $t$ sufficiently close to $t_0$.
\end{proposition}

\begin{proof}
 By Proposition \ref{prop:bounded_exhaust}, we may choose a negative continuous subharmonic function $\rho_t$ on $\Omega_t$ for each $t$ such that
  $$
  (\delta_t/r_t)^2\le -\rho_t\le (\delta_t/r_t)^{1/2}
  $$
  where $r_t=r_{\Omega_t}$. Clearly, $C^{-1}_0<r_t<C_0$ for some uniform constant $C_0>0$. Set
  $$
  \varepsilon=(C_0\gamma_1|t-t_0|^{\alpha})^{1/2}.
  $$
  Since $\Omega_t$ is $(-\delta_t)-$H\"older continuous of order $\alpha$, there exists $\gamma_1\gg 1$ such that
  $$
   \{\delta_t> \gamma_1\,|t-t_0|^\alpha\}=\{\delta_{t_0}> |t-t_0|^\alpha\}=:\Omega_{t_0}^\varepsilon
  $$
  provided $|t-t_0|\le \eta\ll1$.
    Thus
$$
 \{-\rho_t> \varepsilon\}\subset \{\delta_t> \gamma_1\,|t-t_0|^\alpha\}
\subset\Omega_{t_0}^\varepsilon.
$$
 Without loss of generality, we may assume that $S_{t_0}\subset  \{-\rho_t> 2\varepsilon\}$. Fix $w\in S_{t_0}$ for a moment. Let $g_{t_0,\varepsilon}$ denote the Green function of $\Omega_{t_0}^\varepsilon$. Set
$$
b=\inf_{\{\rho_t=-2\varepsilon\}} g_{t_0,\varepsilon}(\cdot,w)
$$
and
$$
\varrho_t=b\cdot \frac{\log(-\rho_t+\varepsilon)-\log 2\varepsilon}{\log 3/2}.
$$
Since $b<0$, we see that $\varrho_t$ is a subharmonic function on $\Omega_t$ which satisfies $\varrho_t=0$ on $\{\rho_t=-\varepsilon\}$ and $\varrho_t=b$ on $\{\rho_t=-2\varepsilon\}$. Set
$$
\psi=\left\{
\begin{array}{ll}
 g_{t_0,\varepsilon}(\cdot,w) & -\rho_t>2\varepsilon\\
 \max\{g_{t_0,\varepsilon}(\cdot,w),\varrho_t\} & \varepsilon\le-\rho_t\le 2\varepsilon\\
 \varrho_t & -\rho_t<\varepsilon.
\end{array}
\right.
$$
It follows that $\psi$ is a well-defined subharmonic function on $\Omega_t$ which has a logarithmic pole at $w$ and an upper bound $b\cdot\frac{\log1/2}{\log 3/2}$. By the well-known extremal property of the Green function, we obtain
\begin{eqnarray*}
g_t(z,w) & \ge & \psi(z)-b\cdot\frac{\log1/2}{\log 3/2}=g_{t_0,\varepsilon}(z,w)-b\cdot\frac{\log1/2}{\log 3/2}\\
& \ge & g_{t_0}(z,w)-b\cdot\frac{\log1/2}{\log 3/2}
\end{eqnarray*}
for all $z\in S_{t_0}$. It remains to estimate $b$. Fix $R>\sup\, {\rm diam}(\Omega_t)$. We may choose positive constants $C_1,C_2$ independent of $t$ such that
$$
\log |\cdot-w|/2R \ge -C_1\ \ \ {\rm if\ \ \ } \rho_t = - C_2,
$$
provided $|t-t_0|\le \eta\ll1$.
Thus
$$
\varphi=\left\{
\begin{array}{ll}
 \log |\cdot-w|/2R & {\rm on\ } \Omega_{t_0}^{\varepsilon_0}\\
 \max\{ \log |\cdot-w|/2R,\frac{C_1}{C_2}\rho_t\} & {\rm on\ } \Omega_{t_0}^{\varepsilon}\backslash \{-\rho_t>-C_2\}
\end{array}
\right.
$$
gives a subharmonic function on $\Omega_{t_0}^\varepsilon$ with a logarithmic pole at $w$, so that
$$
g_{t_0,\varepsilon}(z,w)\ge \frac{C_1}{C_2}\rho_t(z)=-2C_1C_2^{-1}\varepsilon
$$
for all $z$ with $\rho_t(z)=-2\varepsilon$ and $\varepsilon\ll 1$. Thus $b\ge -{\rm const.} \varepsilon$ and
$$
g_t(z,w)\ge g_{t_0}(z,w)-{\rm const.}\varepsilon\ge g_{t_0}(z,w)-{\rm const.} |t-t_0|^{\alpha/2}
$$
for any $z,w\in S_{t_0}$. Similarly, we may verify that
$$
g_{t_0}(z,w)\ge g_{t}(z,w)-{\rm const.} |t-t_0|^{\alpha/2}.
$$
\end{proof}

 \begin{proof}[Proof of Theorem \ref{th:simply-connected}]
    Fix $t_0\in (-1,1)$ for a moment. We may choose a positive number $\varepsilon_0$ such that the disc $\Delta_{2\varepsilon_0}(\zeta)\subset \Omega_t$ for all $\zeta\in S_{t_0}$ and all $t$ sufficiently close to $t_0$. Set $h_t(z,w)=g_t(z,w)-\log |z-w|$ for all $z,w\in \Omega_t$. Clearly,  $h_t(z,w)$ is harmonic in $z$ and $w$ respectively. By Proposition \ref{prop:GreenStability}, we have
    $$
    |h_t(z,w)-h_{t_0}(z,w)|\le {\rm const.} |t-t_0|^{\alpha/2}
    $$
    for all $z,w\in S'_{t_0}=\{z:{\rm dist\,}(z,S_{t_0})\le \varepsilon_0\}$.
       Fix $\xi,\zeta\in S_{t_0}$ for a moment. The Poisson formula asserts
   $$
   h_t(z,w)=\frac1{4\pi^2} \int_0^{2\pi} \int_0^{2\pi} h_t(\xi+\varepsilon_0 e^{i\theta},\zeta+\varepsilon_0 e^{i\vartheta}) \frac{\varepsilon_0^2-|z-\xi|^2}{|\varepsilon_0 e^{i\theta}-(z-\xi)|^2}\frac{\varepsilon_0^2-|w-\zeta|^2}{|\varepsilon_0 e^{i\vartheta}-(w-\zeta)|^2}d\theta d\vartheta.
   $$
   We conclude the proof by using the following famous formula of Schiffer \cite{Schiffer}:
   $$
   K_t(z,w)=\frac2{\pi} \frac{\partial^2 h_t(z,w)}{\partial z\partial\bar{w}}.
   $$
  \end{proof}

 \bigskip

{\bf Acknowledgement.} The author is grateful to Xu Wang for several valuable comments. Part of this work was done during the author's stay at the ESI of Vienna University in November 2015.  He would like to thank the support of the institute.


\begin{thebibliography}{99}

\bibitem{Berndtsson97} B. Berndtsson, {\it Uniform estimates with weights for the $\bar{\partial}-$equation}, J. Geom. Anal. {\bf 7} (1997), 195--215.

\bibitem{Berndtsson01}
----------, {\it Weighted estimates for the $\bar{\partial}$-equation}, Complex Analysis and Complex Geometry (J. D. McNeal eds.),  de Gruyter, pp. 43--57, 2001.

 \bibitem{BerndtssonSubharmonic} ----------, {\it Subharmonicity properties of the Bergman kernel and some other functions associated to pseudoconvex domains}, Ann. Inst. Fourier (Grenoble) {\bf 56} (2006), 1633--1662.

    \bibitem{BerndtssonCurvature} ----------, {\it Curvature of vector bundles associated to holomorphic fibrations}, Ann. of  Math. {\bf 169} (2009), 531--560.

    \bibitem{BerndtssonOpenness}----------, {\it The openness conjecture for plurisubharmonic functions}, arXiv:1305.5781.

         \bibitem{BlockCRvsMA} Z. Blocki, {\it Cauchy-Riemann meet Monge-Amp$\grave{e}$re}, Bull. Math. Sci. {\bf 4} (2014), 433--480.

         \bibitem{DemaillyBook} J.-P. Demailly, Complex Analytic and Differential Geometry, available at Demailly's home page.

         \bibitem{DemaillyKollar} J.-P. Demailly and J. Koll\'ar, {\it Semi-continuity of complex singularity exponents and K\"ahler-Einstein metrics on Fano orbifolds}, Ann. Scient. \'Ec. Norm. Sup. {\bf 34} (2001), 525--556.

 \bibitem{DiederichFornaess77}
K. Diederich and J. E. Fornaess, {\it Pseudoconvex domains: bounded strictly plurisubharmonic exhaustion functions}, Invent. Math. {\bf 39} (1977), 129--141.

 \bibitem{DidOhsParameter} K. Diederich and T. Ohsawa, {\it On the parameter dependence of solutions to the $\bar{\partial}-$equation}, Math. Ann. {\bf 289} (1991), 581--588.

  \bibitem{DonnellyFefferman} H. Donnelly and C. Fefferman, \emph{$L^2$-cohomology and index theorem for the Bergman metric}, Ann. of Math. \textbf{118} (1983), 593--618.

  \bibitem{FarveJonsson} C. Farve and M. Jonsson, {\it Valuations and multiplier ideals}, J. Amer. Math. Soc. {\bf 18} (2005), 655--684.

\bibitem{GreeneKrantz82} R. E. Greene and St. G. Krantz, {\it Deformation of complex structures, estimates for the $\bar{\partial}$ equation and stability of the Bergman kernel}, Adv. Math. {\bf 43} (1982), 1--86.

    \bibitem{GuanZhou} Q. Guan and X. Zhou, {\it Strong openness conjecture for plurisubharmonic functions}, arXiv:1311.3781.

  \bibitem{Hamilton79} R. S. Hamilton, {\it Deformation of complex structures on manifolds with boundary II: Families of non-coercive boundary value problems}, J. Diff. Geom. {\bf 14} (1979), 409--473.

\bibitem{HarringtonPSH} P. S. Harrington, {\it The order of plurisubharmonicity on pseudoconvex domains with Lipschitz boundaries}, Math. Res. Lett. {\bf 14} (2007), 485--490.

\bibitem{Heip} P. H. Heip, {\it The weighted log canonical threshold}, C. R. Math. Acad. Sci. Paris {\bf 352} (2014), 283--288.

\bibitem{HormanderBook}  L. H\"ormander, An introduction to Complex Analysis in Several
Variables, North Holland 1990.

\bibitem{JpnssonMustata} M. Jonsson and M. Mustata, {\it An algebraic approach to the openness conjecture of Demailly and Koll\'ar}, arXiv:1205.4273.

\bibitem{Lempert} L. Lempert, {\it Modules of square integrable holomorphic germs}, arXiv:1404.0407.

\bibitem{MaitaniYamaguchi} F. Maitani and H. Yamaguchi, {\it Variations of Bergman metrics on Riemann surfaces}, Math. Ann. {\bf 330} (2004), 477--489.

\bibitem{PhongSturm} D. H. Phong and J. Sturm, {\it Algebraic estimates, stability of local zeta functions, and uniform estimates for distribution functions}, Ann. of Math. {\bf 152} (2000), 277--329.

    \bibitem{Schiffer} M. Schiffer, {\it The kernel function of an orthonormal system}, Duke Math. J. {\bf 13} (1946), 629--540.

\bibitem{TsujiBook} M. Tsuji, Potential Theory in Modern Function Theory,
Maruzen Co., LTD. Tokyo, 1959.

\bibitem{Varchenko} A. N. Varchenko, {\it Newton polyhedra and estimation of occillating integrals}, Funct. Anal. Appl. {\bf 10} (1976), 175--196.
\end{thebibliography}
\end{document}